\documentclass[11pt]{article}

\usepackage[latin1]{inputenc}
\usepackage[english]{babel}
\usepackage{here,enumerate,graphicx}
\usepackage{amsmath,amssymb,amscd,amsthm,mathrsfs}
\usepackage[colorlinks=true,linkcolor=blue,urlcolor=blue]{hyperref}

\usepackage[flushmargin]{footmisc}
\newcommand\blfootnote[1]{%
	\begingroup
	\renewcommand\thefootnote{}\footnote{#1}%
	\addtocounter{footnote}{-1}%
	\endgroup}

\theoremstyle{plain}
\newtheorem{theorem}{Theorem}[section]
\newtheorem{proposition}[theorem]{Proposition}

\newtheorem{lemma}[theorem]{Lemma}

\theoremstyle{definition}

\newtheorem{remark}[theorem]{Remark}
\newtheorem{example}[theorem]{Example}

\newtheorem{problem}[theorem]{Problem}

\numberwithin{equation}{section}

\newcommand{\NN}{\mathbb{N}}
\newcommand{\QQ}{\mathbb{Q}}
\newcommand{\RR}{\mathbb{R}}
\newcommand{\CC}{\mathbb{C}}
\newcommand{\II}{\mathbb{I}}

\newcommand{\Bc}{\mathcal{B}}
\newcommand{\Cc}{\mathcal{C}}
\newcommand{\Fc}{\mathcal{F}}
\newcommand{\Kc}{\mathcal{K}}
\newcommand{\Pc}{\mathcal{P}}
\newcommand{\Tc}{\mathcal{T}}
\newcommand{\Vc}{\mathcal{V}}
\newcommand{\cl}{\overline}
\newcommand{\diam}{\operatorname{diam}}
\newcommand{\res}{\arrowvert}
\newcommand{\eps}{\varepsilon}

\newcommand{\com}{\overline}%{\bar} % Nos decidimos entre "{\bar}" y "{\overline}"
\newcommand{\fuz}{\hat}%{\widehat} % Nos decidimos entre "{\hat}" y "{\widehat}"
\newcommand{\eend}{\operatorname{end}}
\newcommand{\send}{\operatorname{send}}

\begin{document}
\begin{center}
	\begin{LARGE}
		{\bf Topological dynamics for the endograph metric II: Extremely radical properties}
	\end{LARGE}
\end{center}

\begin{center}
	\begin{Large}
		Antoni L\'opez-Mart\'inez\blfootnote{\textbf{Journal-ref}: Journal of Inequalities and Applications, Volume 2026, article number 77, (2026).\\ \textbf{DOI}: \href{https://doi.org/10.1186/s13660-026-03489-6}{https://doi.org/10.1186/s13660-026-03489-6}}
	\end{Large}
\end{center}

\vspace*{-0.1in}

\begin{abstract}
	Given a dynamical system $(X,f)$ we investigate several topological dynamical properties for its Zadeh extension $(\mathcal{F}(X),\hat{f})$ endowed with the endograph metric $d_{E}$. In particular, we prove that for some contractive and expansive properties, for chain recurrence, chain transitivity and chain mixing, and for the shadowing property, the endograph metric behaves in an extremely radical way. Our results not only resolve certain open questions in the existing literature, but also yield completely new outcomes concerning the chain-type notions considered and the shadowing property.
\end{abstract}

\vspace*{-0.1in}

\section{Introduction}

This paper is a continuation of \cite{Lopez2026_IJFS_topological-I}, where a {\em dynamical system} is a pair $(X,f)$ formed by a continuous map $f:X\longrightarrow X$ acting on a metric space~$(X,d)$. As in \cite{Lopez2026_IJFS_topological-I}, the main objective of this paper is comparing the dynamical properties presented by a dynamical system $(X,f)$ and those presented by the induced systems $(\Kc(X),\com{f})$ and $(\Fc(X),\fuz{f})$, where $\com{f}$ is the usual {\em hyperextension} of $f$ acting on the space $\Kc(X)$ formed by the non-empty compact subsets of $X$, and where $\fuz{f}$ is the so-called {\em Zadeh extension} or {\em fuzzification} of $f$ acting on the space $\Fc(X)$ formed by the normal fuzzy sets of $X$.

Looking for the dynamical properties of $(\Kc(X),\com{f})$ can be considered a classical topic (see \cite{Banks2005_CSF_chaos,BauerSig1975_MM_topological,Peris2005_CSF_set-valued}), but the system $(\Fc(X),\fuz{f})$ has appeared much more recently in the literature and many questions and problems remain open for it (see \cite{AlvarezLoPe2025_FSS_recurrence,BartollMaPeRo2022_AXI_orbit,JardonSan2021_FSS_expansive,JardonSan2021_IJFS_sensitivity,JardonSanSan2020_FSS_some,JardonSanSan2020_MAT_transitivity,MartinezPeRo2021_MAT_chaos}). In addition, although the space $\Kc(X)$ is usually only endowed with the {\em Hausdorff metric} $d_H$, the space $\Fc(X)$ has been endowed with several different metrics such as the {\em supremum metric} $d_{\infty}$, the {\em Skorokhod metric} $d_{0}$, and the {\em sendograph} and {\em endograph metrics}~$d_{S}$ and $d_{E}$. Following \cite{JardonSan2021_FSS_expansive,JardonSan2021_IJFS_sensitivity,JardonSanSan2020_MAT_transitivity,Lopez2026_IJFS_topological-I}, for each metric $\rho \in \{ d_{\infty} , d_{0} , d_{S} , d_{E} \}$ we will denote the metric space $(\Fc(X),\rho)$ by $\Fc_{\infty}(X)$, $\Fc_{0}(X)$, $\Fc_{S}(X)$ and $\Fc_{E}(X)$ respectively, and the symbols $\tau_{\infty}$, $\tau_{0}$, $\tau_{S}$ and $\tau_{E}$ will stand for the respective induced topologies on $\Fc(X)$.

In the previous work \cite{Lopez2026_IJFS_topological-I} we proved that, for several dynamical notions related to topological transitivity, topological recurrence, point-transitivity, Devaney chaos and the specification property, then the systems $(\Kc(X),\com{f})$, $(\Fc_{\infty}(X),\fuz{f})$, $(\Fc_{0}(X),\fuz{f})$, $(\Fc_{S}(X),\fuz{f})$ and $(\Fc_{E}(X),\fuz{f})$ behave exactly in the same way: if $\Pc$ is one of the above mentioned properties, then \textbf{all} of these dynamical systems present such a property $\Pc$ whenever \textbf{any} of them does. Actually, the implication\\[-17.5pt]
\begin{enumerate}[--]
	\item ``\textit{if the system $(\Kc(X),\com{f})$ has property $\Pc$, then $(\Fc_{\infty}(X),\fuz{f})$ has property $\Pc$}'',\\[-17.5pt]
\end{enumerate}
was already proved in the literature for most of the properties $\Pc$ considered in \cite{Lopez2026_IJFS_topological-I}. Then, since each property $\Pc$ considered in \cite{Lopez2026_IJFS_topological-I} depends on ``\textit{certain conditions that \textbf{all open subsets} fulfill}'', and since the inclusions $\tau_{E} \subset \tau_{S} \subset \tau_{0} \subset \tau_{\infty}$ always hold, what we did in \cite{Lopez2026_IJFS_topological-I} was proving the left implication\\[-17.5pt]
\begin{enumerate}[--]
	\item ``\textit{if the system $(\Fc_{E}(X),\fuz{f})$ has property $\Pc$, then $(\Kc(X),\com{f})$ has property $\Pc$}''.
\end{enumerate}

Contrary to the situation that we had in \cite{Lopez2026_IJFS_topological-I}, in this paper we focus on dynamical properties for which the system $(\Fc_{E}(X),\fuz{f})$ exhibits an extremely radical behaviour. In particular, we address the {\em contractive}, {\em expansive}, {\em expanding} and {\em positively expansive} properties (see \cite{JardonSan2021_FSS_expansive,JardonSanSan2020_FSS_some,WuZhangChen2020_FSS_answers}), as well as the notions of {\em chain recurrence}, {\em chain transitivity} and {\em chain mixing} (see \cite{KhanKu2013_FEJDS_recurrence}), and the {\em finite} and {\em full} versions of the so-called {\em shadowing property} (see \cite{BartollMaPeRo2022_AXI_orbit}). The definitions of these properties seem to rely much more on the specific metric used than on the topology that such a metric generates. This is probably why these properties behave in an extremely pronounced manner for $d_{E}$, as we summarize in the following list of main results obtained:
\begin{enumerate}[--]
	\item When a metric space $(X,d)$ has more than one point then the corresponding space $\Fc_{E}(X)$ has no isolated points, regardless of whether $(X,d)$ itself has isolated points; see Lemma~\ref{Lem:isolated.points} below. We consider that this is an extremely radical behaviour since, as we also prove below, the rest of metric spaces $(\Kc(X),d_H)$, $\Fc_{\infty}(X)$, $\Fc_{0}(X)$ and $\Fc_{S}(X)$ have isolated points if and only if so does $(X,d)$.
	
	\item Given a dynamical system $(X,f)$, then the systems $(\Fc_{E}(X),\fuz{f})$, $(\Fc_{S}(X),\fuz{f})$ and $(\Fc_{0}(X),\fuz{f})$ are contractive if and only if the map $f$ is constant; see statement (a) of Theorem~\ref{The:contraexpansive} below. This result reproves \cite[Example~7]{WuZhangChen2020_FSS_answers} in a shorter way, it generalizes \cite[Theorem~8]{WuZhangChen2020_FSS_answers}, and it completely solves both \cite[Problems~5.8~and~5.12]{JardonSanSan2020_FSS_some}. In particular, such a result exhibits an extremely radical behaviour for the endograph metric (the dynamics of a constant map are far from being interesting).
	
	\item Given a dynamical system $(X,f)$, then the systems $(\Fc_{E}(X),\fuz{f})$, $(\Fc_{S}(X),\fuz{f})$ and $(\Fc_{0}(X),\fuz{f})$ are expansive, expanding and positively expansive if and only if the set $X$ has exactly one point; see statement (b) of Theorem~\ref{The:contraexpansive} below. This result solves both \cite[Problems~3.6~and~4.3]{JardonSan2021_FSS_expansive}, reproves in a shorter way \cite[Examples~3.5,~3.9~and~4.6]{JardonSan2021_FSS_expansive}, and shows that some of the implications proved in~\cite[Theorems~3.1~and~4.1]{JardonSan2021_FSS_expansive}, but also the results \cite[Propositions~3.8,~4.4~and~4.5]{JardonSan2021_FSS_expansive}, are practically void (such results only apply when $X$ is a singleton, which is rarely assumed).
	
	\item Given a dynamical system $(X,f)$, then the system $(\Fc_{E}(X),\fuz{f})$ is chain recurrent, chain transitive and chain mixing if and only if the map $f$ has dense range (see Theorem~\ref{The:chains.E} below). Again, this seems an extreme behaviour because having dense range is a commonly assumed dynamical property.
	
	\item Given a dynamical system $(X,f)$ for which $f$ has dense range, then $(\Fc_{E}(X),\fuz{f})$ cannot have the finite shadowing property when $(X,f)$ is not topologically mixing (see Theorem~\ref{The:shadowing.E}). This result can be used to show that $(\Fc_{E}(X),\fuz{f})$ does not necessarily have the finite shadowing property when the system $(X,f)$ is contractive (see Example~\ref{Exa_2:connected}), but $(\Fc_{E}(X),\fuz{f})$ does have the full shadowing property when $(X,f)$ is contractive and $f^k(X)$ is bounded for some $k \in \NN$ (see Theorem~\ref{The:contractions}).
\end{enumerate}

The rest of the paper is organized as follows. In Section~\ref{Sec_2:notation} we introduce the general background on the hyperspaces of compact and fuzzy sets and we prove Lemma~\ref{Lem:isolated.points} mentioned above. In Section~\ref{Sec_3:contraexpansive} we look for the {\em contractive}, {\em expansive}, {\em expanding} and {\em positively expansive} properties. In Section~\ref{Sec_4:chains} we focus on the notions of {\em chain recurrence}, {\em chain transitivity} and {\em chain mixing}. In Section~\ref{Sec_5:shadowing} we consider the so-called {\em shadowing property} together with its {\em finite} version. Finally, in Section~\ref{Sec_6:conclusions}, we summarize the results obtained in both \cite{Lopez2026_IJFS_topological-I} and this paper, establishing some possible future lines of research.

\section{Notation and general background}\label{Sec_2:notation}

In this section we introduce the notation used along the paper, which will be completely similar to that used in \cite{Lopez2026_IJFS_topological-I}. In particular, we recall the definition of the extensions $(\Kc(X),\com{f})$ and $(\Fc(X),\fuz{f})$, for a given dynamical system $(X,f)$, and we prove Lemma~\ref{Lem:isolated.points} mentioned above. From now on let $\NN$ be the set of strictly positive integers, set $\NN_0 := \NN \cup \{0\}$, and let $\II$ be the unit interval $[0,1]$.

\subsection[The extended systems (K(X),f) and (F(X),f)]{The extended systems $(\Kc(X),\com{f})$ and $(\Fc(X),\fuz{f})$}\label{SubSec_2.1:extensions}

Given a metric space $(X,d)$ we will denote by $\Bc_d(x,\eps) \subset X$ the open $d$-ball centred at $x \in X$ and of radius $\eps>0$. Moreover, we will consider the spaces $\Cc(X) :=\left\{ C \subset X \ ; \ C \text{ is a non-empty closed set} \right\}$ and $\Kc(X) := \left\{ K \subset X \ ; \ K \text{ is a non-empty compact set} \right\}$. It is clear that $\Kc(X) \subset \Cc(X)$. Given two sets $C_1,C_2 \in \Cc(X)$, the {\em Hausdorff distance} between them is
\[
d_H(C_1,C_2) := \max\left\{ \sup_{x_1 \in C_1} d(x_1,C_2) , \sup_{x_2 \in C_2} d(x_2,C_1) \right\}.
\]
This value may be infinite if either $C_1$ or $C_2$ are not compact. However, in $\Kc(X)$ we have that the map~$d_H:\Kc(X)\times\Kc(X)\longrightarrow[0,+\infty[$ defined as above is a metric, called the {\em Hausdorff metric}, and every continuous map $f:X\longrightarrow X$ induces a $d_H$-continuos map $\com{f}:\Kc(X)\longrightarrow\Kc(X)$, defined as~$\com{f}(K) := f(K) = \{ f(x) \ ; \ x \in K \}$ for each $K \in \Kc(X)$. Thus, given a dynamical system $(X,f)$ we have reached the first extended system that we will consider through this paper, namely $(\Kc(X),\com{f})$.

We will denote by $\Bc_H(K,\eps) \subset \Kc(X)$ the open $d_H$-ball centred at $K \in \Kc(X)$ and of radius $\eps>0$. The collection of these balls provide a basis for the topology induced by $d_H$. This topology is known to coincide with the so-called {\em Vietoris topology}, whose basic open sets are the sets of the form
\[
\Vc(U_1,U_2,...,U_N) :=  \left\{ K \in \Kc(X) \ ; \ K \subset \bigcup_{j=1}^N U_j \text{ and } K\cap U_j\neq\varnothing \text{ for all } 1\leq j\leq N \right\},
\]
where $N \in \NN$ and $U_1,U_2,...,U_N$ are non-empty open subsets of $X$. Given $Y \subset X$ and $\eps\geq 0$ we will write $Y+\eps := \{ x \in X \ ; \ d(x,y) \leq \eps \text{ for some } y \in Y \}$, and the following are well-known facts:

\begin{proposition}\label{Pro:Hausdorff}
	Let $(X,d)$ be a metric space, $\eps\geq 0$, and let $C_1,C_2,C_3,C_4 \in \Cc(X)$. Then:
	\begin{enumerate}[{\em(a)}]
		\item We have that $d_H(C_1,C_2) \leq \eps$ if and only if $C_1 \subset C_2+\eps$ and $C_2 \subset C_1+\eps$.
		
		\item We always have that $d_H(C_1\cup C_2, C_3 \cup C_4) \leq \max\{ d_H(C_1,C_3) , d_H(C_2,C_4) \}$.
	\end{enumerate}
\end{proposition}

We refer the reader to \cite{IllanesNad1999_book_hyperspaces} for a detailed study of $\Cc(X)$ and $\Kc(X)$, but let us now move into the fuzzy content: a {\em fuzzy~set} on the metric space $(X,d)$ will be a function $u:X\longrightarrow\II$, where the value $u(x) \in \II$ denotes the degree of membership of $x$ in $u$. For such a $u$, its {\em $\alpha$-level} is the set
\[
u_{\alpha} := \{ x \in X \ ; \ u(x) \geq \alpha \} \ \text{ for each } \alpha \in \ ]0,1] \quad \text{ and } \quad u_0 := \cl{\{ x \in X \ ; \ u(x)>0 \}}.
\]
We will denote by $\Fc(X)$ the space of normal fuzzy sets of $(X,d)$, i.e.\ the space formed by the fuzzy sets~$u$ that are upper-semicontinuous functions and such that~$u_0$ is compact and~$u_1$ is non-empty. For each subset $Y \subset X$ we denote by $\chi_Y:X\longrightarrow\II$ the characteristic function on~$Y$, which belongs to $\Fc(X)$ if and only if $Y$ is non-empty and compact. The piecewise-constant function $u := \max_{1\leq l\leq N} (\alpha_l \cdot \chi_{K_l})$ also belongs to $\Fc(X)$ whenever $0 < \alpha_1 < \alpha_2 < ... < \alpha_N = 1$ and $K_1,K_2,...,K_N \in \Kc(X)$. Note that $u=\max_{1\leq l\leq N} (\alpha_l \cdot \chi_{K_l})$ fulfills that $u_{\alpha} = \bigcup \left\{ K_l \ ; \ 1\leq l\leq N \text{ with } \alpha_l \geq \alpha \right\}$ for each $\alpha \in [0,1]$, and that
\[
u_{\alpha_l} = \bigcup_{j \geq l} K_j \quad \text{ for each } 1\leq l\leq N.
\]
The next result will allow to get piecewise-constant functions near every fuzzy set (see \cite{JardonSanSan2020_FSS_some,JardonSanSan2020_MAT_transitivity,MartinezPeRo2021_MAT_chaos}):

\begin{lemma}\label{Lem:eps.pisos}
	Let $(X,d)$ be a metric space. Given any set $u \in \Fc(X)$ and $\eps>0$ there exist numbers $0 = \alpha_0 < \alpha_1 < \alpha_2 < ... < \alpha_N = 1$ such that $d_H(u_{\alpha}, u_{\alpha_{l+1}})<\eps$ for all $\alpha \in \ ]\alpha_l, \alpha_{l+1}]$ and $0\leq l\leq N-1$. In particular, since $d_H(u_{\alpha}, u_{\alpha_1})<\eps$ for every $\alpha \in \ ]0, \alpha_1]$, we also have that $d_H(u_0,u_{\alpha_1})\leq\eps$.
\end{lemma}

Given a dynamical system $(X,f)$, it follows from \cite[Propositions~3.1 and 4.9]{JardonSanSan2020_FSS_some} that there exists a well-defined map $\fuz{f}:\Fc(X)\longrightarrow\Fc(X)$, called the {\em Zadeh extension} of $f$, which maps each normal fuzzy set $u \in \Fc(X)$ to the following normal fuzzy set:
\begin{equation*}
\fuz{f}(u):X\longrightarrow\II \quad \text{ where } \quad \left[\fuz{f}(u)\right](x) :=
\left\{
\begin{array}{lcc}
	\sup\{ u(y) \ ; \ y \in f^{-1}(\{x\}) \}, & \text{ if } f^{-1}(\{x\}) \neq \varnothing, \\[5pt]
	0, & \text{ if } f^{-1}(\{x\}) = \varnothing.
\end{array}
\right.
\end{equation*}
Moreover, the following properties are well-known (see \cite{JardonSanSan2020_FSS_some,JardonSanSan2020_MAT_transitivity,RomanChal2008_CSF_some} for details):

\begin{proposition}\label{Pro:fuz{f}}
	Let $f:X\longrightarrow X$ be a continuous map acting on a metric space $(X,d)$, $u \in \Fc(X)$, $\alpha \in \II$, $n \in \NN_0$, and $K \in \Kc(X)$. Then:
	\begin{enumerate}[{\em(a)}]
		\item $\left[ \fuz{f}(u) \right]_{\alpha} = f(u_{\alpha}) = \com{f}(u_{\alpha})$, i.e.\ the $\alpha$-level of the $\fuz{f}$-image coincides with the $\com{f}$-image of the $\alpha$-level;
		
		\item $\left( \fuz{f} \right)^n = \widehat{f^n}$, i.e.\ the composition of $\fuz{f}$ with itself $n$ times coincides with the fuzzification of $f^n$;
		
		\item $\fuz{f}\left( \chi_K \right) = \chi_{f(K)} = \chi_{\com{f}(K)}$, i.e.\ the fuzzification $\fuz{f}$ is an extension of the hyperextension $\com{f}$.
	\end{enumerate}
\end{proposition}

To talk about the continuity of the Zadeh extension $\fuz{f}:\Fc(X)\longrightarrow\Fc(X)$, we will endow the space of normal fuzzy sets $\Fc(X)$ with the four different metrics mentioned at the Introduction.

\subsection[Four metrics for the space F(X)]{Four metrics for the space $\Fc(X)$}\label{SubSec_2.2:metrics}

Let $(X,d)$ be a metric space. Thus, the {\em supremum metric} $d_{\infty}:\Fc(X)\times\Fc(X)\longrightarrow[0,+\infty[$ is defined for each pair $u,v \in \Fc(X)$ as
\[
d_{\infty}(u,v) := \sup_{\alpha\in\II} d_H(u_{\alpha},v_{\alpha}),
\]
where $d_H$ is the Hausdorff metric on $\Kc(X)$. Such a metric has also been called the {\em level-wise metric} in many references (see \cite{Kupka2011_IS_on} and the references included there). In the second place we will consider the {\em Skorokhod metric} $d_{0}:\Fc(X)\times\Fc(X)\longrightarrow[0,+\infty[$, which is defined for each pair $u,v \in \Fc(X)$ as
\[
d_{0}(u,v) := \inf\left\{ \eps>0 \ ; \ \text{there is } \xi \in \Tc \text{ such that } \sup_{\alpha\in\II} |\xi(\alpha)-\alpha| \leq \eps \text{ and } d_{\infty}(u,\xi\circ v) \leq \eps \right\},
\]
where $\Tc$ is the set of strictly increasing homeomorphisms of the unit interval $\xi:\II\longrightarrow\II$. This metric was introduced, for the case in which $(X,d)$ is an arbitrary metric space, in the 2020 paper \cite{JardonSanSan2020_FSS_some}.

To talk about the {\em endograph} and {\em sendograph metrics} we need to consider an auxiliary metric $\com{d}$ on the product space $X\times\II$, which is defined as follows:
\[
\com{d}\left((x,\alpha),(y,\beta)\right) := \max\left\{ d(x,y) , |\alpha - \beta | \right\} \quad \text{ for each pair } (x,\alpha),(y,\beta) \in X\times\II.
\]
Given now any fuzzy set $u \in \Fc(X)$, the {\em endograph of $u$} is defined as the set
\[
\eend(u) := \left\{ (x,\alpha) \in X\times\II \ ; \ u(x) \geq \alpha \right\},
\]
and the {\em sendograph of $u$} is defined as $\send(u) := \eend(u) \cap (u_0\times\II)$. Then, for each pair $u,v \in \Fc(X)$ the~{\em endograph metric}~$d_{E}(u,v)$ on~$\Fc(X)$ is the Hausdorff distance $\com{d}_H(\eend(u),\eend(u))$ on $\Cc(X\times\II)$, and the {\em sendograph metric} $d_S(u,v)$ on $\Fc(X)$ is the Hausdorff metric $\com{d}_H(\send(u),\send(v))$ on $\Kc(X\times\II)$. Although $d_{E}$ is defined as a distance of closed but not necessarily compact sets on $X\times\II$, the fact that the supports of the fuzzy sets $u,v \in \Fc(X)$ are compact implies that $d_{E}(u,v)$ is well-defined.

Following \cite{Lopez2026_IJFS_topological-I} we will denote by $\Bc_{\infty}(u,\eps)$, $\Bc_{0}(u,\eps)$, $\Bc_{S}(u,\eps)$ and $\Bc_{E}(u,\eps)$ the open balls centred at $u \in \Fc(X)$ and of radius $\eps>0$ for each of the metrics $d_{\infty}$, $d_{0}$, $d_{S}$ and $d_{E}$. Moreover, and as mentioned in the Introduction, for each metric $\rho \in \{ d_{\infty} , d_{0} , d_{S} , d_{E} \}$ we will denote the metric space $(\Fc(X),\rho)$ by $\Fc_{\infty}(X)$, $\Fc_{0}(X)$, $\Fc_{S}(X)$ and $\Fc_{E}(X)$ respectively, and $\tau_{\infty}$, $\tau_{0}$, $\tau_{S}$ and $\tau_{E}$ will stand for the respective induced topologies on $\Fc(X)$. These metrics come from the general theory of Spaces of Fuzzy Sets and each of them has its own role and importance (see \cite[Section~2]{Lopez2026_IJFS_topological-I}). The study of the relationships between these metrics has been an important topic addressed in the literature and we recall here the main relations among them (see \cite{JardonSan2021_FSS_expansive,JardonSan2021_IJFS_sensitivity,JardonSanSan2020_FSS_some,Lopez2026_IJFS_topological-I} and the references therein for more details):

\begin{proposition}\label{Pro:fuzzy.metrics}
	Let $(X,d)$ be a metric space, $u,v \in \Fc(X)$, $K,L \in \Kc(X)$ and $x \in X$. Then:
	\begin{enumerate}[{\em(a)}]
		\item $d_{E}(u,v) \leq d_{S}(u,v) \leq d_{0}(u,v) \leq d_{\infty}(u,v)$ and hence $\tau_{E} \subset \tau_{S} \subset \tau_{0} \subset \tau_{\infty}$.
		
		\item $d_{S}(\chi_{\{x\}},u) = d_{0}(\chi_{\{x\}},u) = d_{\infty}(\chi_{\{x\}},u) = d_H(\{x\},u_0) = \max\{ d(x,y) \ ; \ y \in u_0 \}$.
		
		\item $d_{0}(\chi_{K},u) = d_{\infty}(\chi_{K},u) = \max\{ d_H(K,u_0) , d_H(K,u_1) \}$.
		
		\item $d_{E}(\chi_{K},\chi_{L})=\min\{ d_H(K,L) , 1 \}$ while $d_{S}(\chi_{K},\chi_{L}) = d_{0}(\chi_{K},\chi_{L}) = d_{\infty}(\chi_{K},\chi_{L})=d_H(K,L)$.
		
		\item $d_H(u_0,v_0) \leq d_{S}(u,v)$ and also $\max\{ d_H(u_0,v_0), d_H(u_1,v_1) \} \leq d_{0}(u,v)$.
	\end{enumerate}
\end{proposition}

The next lemma was the key fact used in \cite{Lopez2026_IJFS_topological-I} to get most of the results obtained there, but we will also use it here in Sections~\ref{Sec_4:chains}~and~\ref{Sec_5:shadowing}:

\begin{lemma}[\textbf{\cite[Lemma~2.4]{Lopez2026_IJFS_topological-I}}]\label{Lem:key}
	Let $(X,d)$ be a metric space, and assume that the sets $K \in \Kc(X)$ and $u \in \Fc(X)$ fulfill that $\delta := d_{E}(\chi_K,u) < \tfrac{1}{2}$. Then $d_H(K,u_{\alpha}) \leq \delta$ for every $\alpha \in \ ]\delta,1-\delta]$.
\end{lemma}

Given a dynamical system $(X,f)$, the continuity of the extension $\fuz{f}:(\Fc(X),\rho)\longrightarrow(\Fc(X),\rho)$ for every metric $\rho \in \{ d_{\infty} , d_{0} , d_{S} , d_{E} \}$ is also well-known (see \cite{JardonSanSan2020_FSS_some} and \cite{Kupka2011_IS_on}). Thus, we have reached the second extended system that we will consider through this paper, namely $(\Fc(X),\fuz{f})$. However, before studying the dynamical properties of such a system, let us close this section by proving that the endograph metric $d_{E}$ presents an extreme behaviour with respect to the existence of isolated points:

\begin{lemma}\label{Lem:isolated.points}
	Let $(X,d)$ be a metric space. Hence:
	\begin{enumerate}[{\em(a)}]
		\item The following statements are equivalent:
		\begin{enumerate}[{\em(i)}]
			\item the set $X$ is a singleton;
			
			\item the set $\Kc(X)$ is a singleton;
			
			\item the set $\Fc(X)$ is a singleton.
		\end{enumerate}
		
		\item The following statements are equivalent:
		\begin{enumerate}[{\em(i)}]
			\item the metric space $(X,d)$ has at least one isolated point;
			
			\item the metric space $(\Kc(X),d_H)$ has at least one isolated point;
			
			\item the metric space $\Fc_{\infty}(X)$ has at least one isolated point;
			
			\item the metric space $\Fc_{0}(X)$ has at least one isolated point;
			
			\item the metric space $\Fc_{S}(X)$ has at least one isolated point.
		\end{enumerate}
	
		\item If the set $X$ is not a singleton, then the metric space $\Fc_{E}(X)$ has no isolated points.
	\end{enumerate}
\end{lemma}
\begin{proof}
	(a): It is not hard to check that $X$ is a singleton, say $X=\{x\}$, if and only if $\Kc(X)=\{\{x\}\}$, but also if and only if $\Fc(X)=\{\chi_{\{x\}}\}$.
	
	(b): The equivalence (i) $\Leftrightarrow$ (ii) is well-known (see \cite[Proposition~8.3]{IllanesNad1999_book_hyperspaces}). To prove (i) $\Rightarrow$ (v), assume that $x \in X$ is an isolated point for $(X,d)$ so that there is some $\delta>0$ for which $\{x\}+\delta=\{x\}$. Thus, given any fuzzy set $u \in \Bc_{S}(\chi_{\{x\}},\delta)$ we have that $d_H(\{x\},u_0) = d_{S}(\chi_{\{x\}},u) \leq \delta$, but then we have that $u_{0} \subset \{x\}+\delta = \{x\}$, so that $u_0=\{x\}$ and hence $u=\chi_{\{x\}}$. It follows that the fuzzy set $\chi_{\{x\}}$ is an isolated point for $\Fc_{S}(X)$, which proves (v). The implications (v) $\Rightarrow$ (iv) $\Rightarrow$ (iii) follow from the inclusions $\tau_{S} \subset \tau_{0} \subset \tau_{\infty}$. Finally, to prove (iii) $\Rightarrow$ (i), assume that $(X,d)$ has no isolated points and let us check that, given any $u \in \Fc(X)$ and any $\eps>0$, the set $\Bc_{\infty}(u,\eps)\setminus\{u\}$ is non-empty. Indeed, choose any $x \in u_1$ and consider the open $d$-ball $\Bc_d(x,\eps)$. We thus have two possibilities:
	\begin{enumerate}[--]
		\item \textbf{Case 1}: \textit{There exists some point $y \in \Bc_d(x,\eps) \setminus u_1$}. In this case we can consider the function
		\[
		v:X\longrightarrow\II \quad \text{ with } \quad v := \max\{ u , \chi_{\{y\}} \}.
		\]
		It is not hard to check that $v \in \Fc(X)$, that $u \neq v$, and that $v \in \Bc_{\infty}(u,\eps)\setminus\{u\}$.
		
		\item \textbf{Case 2}: \textit{We have the inclusion $\Bc_d(x,\eps) \subset u_1$}. In this case pick any positive value $0<\delta<\eps$ fulfilling that $X\setminus\Bc_d(x,\delta) \neq \varnothing$, which can be done because we are assuming that $(X,d)$ has no isolated points so that $X$ is not a singleton, and consider the function
		\[
		w:X\longrightarrow\II \quad \text{ with } \quad w := u \cdot \chi_{X\setminus \Bc_d(x,\delta)}.
		\]
		It is not hard to check that $w \in \Fc(X)$, that $u \neq w$, and that $w \in \Bc_{\infty}(u,\eps)\setminus\{u\}$.
	\end{enumerate}
	
	(c): Assume that the set $X$ is not a singleton and, given any fuzzy set $u \in \Fc(X)$ and any $\eps>0$, let us prove that $\Bc_{E}(u,\eps)\setminus\{u\}$ is non-empty. We have two possibilities:
	\begin{enumerate}[--]
		\item \textbf{Case 1}: \textit{We have that $u_1 \neq X$}. In this case pick any point $x \in X\setminus u_1$, pick any positive value $0<\delta<\eps$ and let $\alpha := \min\{ u(x)+\delta , 1 \}$. Thus, we can consider the function
		\[
		v:X\longrightarrow\II \quad \text{ with } \quad v := \max\{ u , \alpha\cdot\chi_{\{x\}} \}.
		\]
		It is not hard to check that $v \in \Fc(X)$, that $u \neq v$, and that $v \in \Bc_{E}(u,\eps)\setminus\{u\}$.
		
		\item \textbf{Case 2}: \textit{We have that $u_1 = X$}. In this case pick any point $x \in X = u_1$ and pick any positive value $0<\delta<\eps$ fulfilling that $X\setminus\Bc_d(x,\delta) \neq \varnothing$, which can be done because we are assuming that $X$ is not a singleton. Thus, we can consider the function
		\[
		w:X\longrightarrow\II \quad \text{ with } \quad w := u \cdot \chi_{X\setminus \Bc_d(x,\delta)}.
		\]
		It is not hard to check that $w \in \Fc(X)$, that $u \neq w$, and that $w \in \Bc_{\infty}(u,\eps)\setminus\{u\} \subset \Bc_{E}(u,\eps)\setminus\{u\}$.\qedhere
	\end{enumerate}
\end{proof}

\begin{remark}\label{Rem:path.connected}
	A result stronger than statement (c) of Lemma~\ref{Lem:isolated.points} holds: {\em for any metric space~$(X,d)$, the respective metric space $\Fc_{E}(X)$ is path-connected}. We do not use this fact along the paper, but let us briefly justify it. Actually, to avoid trivialities, let us assume that $X$ and hence the set $\Fc(X)$ have more than one point. Thus, given two distinct fuzzy sets $u,v \in \Fc(X)$ we can find a continuous path from $u$ to $v$ by concatenating four paths: pick $x \in u_1$ and $y \in v_1$, assume that $x \neq y$ since otherwise we can avoid paths $\gamma_2$ and $\gamma_3$, and let $\gamma_1,\gamma_2,\gamma_3,\gamma_4:[0,1]\longrightarrow\Fc(X)$ be defined for each $t \in [0,1]$ as $\gamma_1(t) := \max\{ \chi_{\{x\}} , (1-t)\cdot u \}$, $\gamma_2(t) := \chi_{\{x\}} + t \cdot \chi_{\{y\}}$, $\gamma_3(t) := \chi_{\{y\}} + (1-t) \cdot \chi_{\{x\}}$ and $\gamma_4(t) := \max\{ \chi_{\{y\}} , t \cdot v \}$. It can be checked that $\gamma_1,\gamma_2,\gamma_3$ and $\gamma_4$ are continuous precisely when $\Fc(X)$ is endowed with the endograph metric $d_{E}$. Moreover, we have that $\gamma_1(0)=u$, $\gamma_1(1)=\chi_{\{x\}}=\gamma_2(0)$, $\gamma_2(1)=\chi_{\{x,y\}}=\gamma_3(0)$, $\gamma_3(1)=\chi_{\{y\}}=\gamma_4(0)$ and $\gamma_4(1)=v$. The stated result follows.
\end{remark}

\section{Contractions and expansive properties}\label{Sec_3:contraexpansive}

In this section we focus on the contractive and expansive-type properties considered in \cite{JardonSan2021_FSS_expansive,JardonSanSan2020_FSS_some}. Our results solve \cite[Problems~3.6,~3.7~and~4.3]{JardonSan2021_FSS_expansive} and \cite[Problems~5.8~and~5.12]{JardonSanSan2020_FSS_some}. We have divided the section into three parts: first we contextualize and state our main result (see Theorem~\ref{The:contraexpansive}); then we prove it using Lemma~\ref{Lem:contraexpansive}; and we finally solve \cite[Problem~3.7]{JardonSan2021_FSS_expansive} in the negative with Example~\ref{Exa:expanding}.

\subsection{Definitions and main result}

Following \cite{JardonSan2021_FSS_expansive,JardonSanSan2020_FSS_some}, given a continuous map $f:X\longrightarrow X$ acting on a metric space $(X,d)$, we will say that the corresponding dynamical system $(X,f)$ is:
\begin{enumerate}[--]
	\item {\em contractive} if there exists $\lambda \in [0,1[$ such that $d(f(x),f(y)) \leq \lambda d(x,y)$ for every $x,y \in X$;
	
	\item {\em expansive} if there exists $\lambda>1$ such that $d(f(x),f(y)) \geq \lambda d(x,y)$ for every $x,y \in X$.
\end{enumerate}
These properties come from the theory of Dynamical Systems and they have played a very important role in such a field. Actually, very well-known results such as the Banach fixed-point theorem, or the so-called Walter topologically-stable theorem, justify the importance of the contractive and expansive dynamical systems. Moreover, several variations of the expansive property have been considered in the literature and, following \cite{JardonSan2021_FSS_expansive,RichersonWise2007_TA_positively}, we will say that a dynamical system $(X,f)$ is:
\begin{enumerate}[--]
	\item {\em expanding} if there exist a positive value $\eps>0$ and $\lambda>1$ such that $d(f(x),f(y)) > \lambda d(x,y)$ for every pair of points $x,y \in X$ fulfilling that $0<d(x,y)<\eps$;
	
	\item {\em positively expansive} if there exists a constant $\delta>0$ such that for any pair of distinct points $x \neq y$ in $X$ there is a non-negative integer $n \in \NN_0$ fulfilling that $d(f^n(x),f^n(y)) > \delta$.
\end{enumerate}
It is known that every {\em expansive} system is {\em expanding}, and the notion of {\em positive expansiveness} has been considered on dynamical systems over compact manifolds with boundary. For instance, it was proved in \cite[Theorem~1]{Hirade1990_PAMS_nonexistence} that no compact connected manifold with boundary admits a positively expansive map, and this kind of negative result is similar to the ones we obtain here. Let us recall that, for the systems $(\Kc(X),\com{f})$ and $(\Fc(X),\fuz{f})$, the next facts have been proved in \cite{Barnsley1988_book_fractals,JardonSan2021_FSS_expansive,JardonSanSan2020_FSS_some,MaZhuLu2016_SPlus_some,RomanChal2005_CSF_robinson-chaos}:
\begin{enumerate}[--]
	\item {\em a dynamical system $(X,f)$ is contractive if and only if so is $(\Kc(X),\com{f})$, but also if and only if so is the fuzzy extension $(\Fc_{\infty}(X),\fuz{f})$}, see \cite[Section~III.7]{Barnsley1988_book_fractals} and \cite[Proposition~5.7]{JardonSanSan2020_FSS_some};
	
	\item {\em a dynamical system $(X,f)$ is expansive if and only if so is $(\Kc(X),\com{f})$, but also if and only if so is the fuzzy extension $(\Fc_{\infty}(X),\fuz{f})$}, see \cite[Proposition~3]{RomanChal2005_CSF_robinson-chaos} and \cite[Theorem~5]{MaZhuLu2016_SPlus_some};
	
	\item {\em the system $(\Kc(X),\com{f})$ is expanding (resp.\ positively expansive) if and only if so is $(\Fc_{\infty}(X),\fuz{f})$, and in that case $(X,f)$ is expanding (resp.\ positively expansive)}, see \cite[Theorems~3.1~and~4.1]{JardonSan2021_FSS_expansive};
	
	\item {\em there is a expanding (resp.\ positively expansive) dynamical system $(X,f)$ for which $(\Kc(X),\com{f})$ and hence $(\Fc_{\infty}(X),\fuz{f})$ are not expanding (resp.\ positively expansive)}, see \cite[Examples~3.2~and~4.2]{JardonSan2021_FSS_expansive}.
\end{enumerate}
In this context, it was asked in \cite{JardonSan2021_FSS_expansive,JardonSanSan2020_FSS_some} about the possible contractive, expansive, expanding and positively expansive behaviour of $(\Fc_{0}(X),\fuz{f})$, $(\Fc_{S}(X),\fuz{f})$ and $(\Fc_{E}(X),\fuz{f})$, whenever $(\Kc(X),\com{f})$ presents any of these properties (see \cite[Problems~3.6~and~4.3]{JardonSan2021_FSS_expansive} and \cite[Problems~5.8~and~5.12]{JardonSanSan2020_FSS_some}). The main result of this section shows that it occurs only when the original system $(X,f)$ is trivial:

\begin{theorem}\label{The:contraexpansive}
	Let $f:X\longrightarrow X$ be a continuous map on a metric space $(X,d)$. Hence:
	\begin{enumerate}[{\em(a)}]
		\item The extended dynamical systems $(\Fc_{0}(X),\fuz{f})$, $(\Fc_{S}(X),\fuz{f})$ and $(\Fc_{E}(X),\fuz{f})$ are contractive if and only if the map $f$ is constant.
		
		\item The extended dynamical systems $(\Fc_{0}(X),\fuz{f})$, $(\Fc_{S}(X),\fuz{f})$ and $(\Fc_{E}(X),\fuz{f})$ are expansive, expanding and positively expansive if and only if the set $X$ is a singleton.
	\end{enumerate}
\end{theorem}

\begin{remark}\label{Rem:solved}
	The reader should note that:
	\begin{enumerate}[(1)]
		\item Statement (a) of Theorem~\ref{The:contraexpansive} can be used to reprove \cite[Example~7]{WuZhangChen2020_FSS_answers} in a shorter way, it generalizes \cite[Theorem~8]{WuZhangChen2020_FSS_answers}, and it completely solves both \cite[Problems~5.8~and~5.12]{JardonSanSan2020_FSS_some} by characterizing the systems $(X,f)$ whose extensions $(\Fc_{0}(X),\fuz{f})$, $(\Fc_{S}(X),\fuz{f})$ and $(\Fc_{E}(X),\fuz{f})$ are contractive.
		
		\item Statement (b) of Theorem~\ref{The:contraexpansive} solves both \cite[Problems~3.6~and~4.3]{JardonSan2021_FSS_expansive}, it can be used to reprove in a shorter way \cite[Examples~3.5,~3.9~and~4.6]{JardonSan2021_FSS_expansive}, and it shows that the implications ``(iv) $\Rightarrow$ (iii)'' proved in \cite[Theorems~3.1~and~4.1]{JardonSan2021_FSS_expansive}, but also the results \cite[Propositions~3.8,~4.4~and~4.5]{JardonSan2021_FSS_expansive}, are practically void (such results only apply when $X$ is a singleton, which is rarely assumed).
	\end{enumerate} 
\end{remark}

\subsection{Proof of Theorem~\ref{The:contraexpansive}}

To prove Theorem~\ref{The:contraexpansive} we will need an auxiliary lemma. From now on, given a metric space $(X,d)$ and a subset $Y \subset X$ we will denote by $\diam_d(Y) := \sup\{ d(x,y) \ ; \ x,y \in Y \}$ the {\em $d$-diameter} of $Y$. Note that the set $X$ is a singleton if and only if $\diam_d(X)=0$, and that given any map $f:X\longrightarrow X$ we have that $f$ is constant if and only if $\diam_d(f(X))=0$. With this notation we have the following:

\begin{lemma}\label{Lem:contraexpansive}
	Let $f:X\longrightarrow X$ be a continuous map on a metric space $(X,d)$. Hence:
	\begin{enumerate}[{\em(a)}]
		\item If we have that $\diam_d(f(X))>0$, then for each integer $k > \max\{ 2 , \diam_d(f(X))^{-1} \}$ there exist two normal fuzzy sets $u,u^k \in \Fc(X)$ fulfilling that
		\[
		d_{E}(u,u^k) \leq d_{S}(u,u^k) \leq d_{0}(u,u^k) \leq \tfrac{1}{k}
		\]
		and
		\[
		d_{E}(\fuz{f}(u),\fuz{f}(u^k)) = d_{S}(\fuz{f}(u),\fuz{f}(u^k)) = d_{0}(\fuz{f}(u),\fuz{f}(u^k)) = \tfrac{1}{k}.
		\]
		
		\item If we have that $\diam_d(X)>0$, then for each integer $k > \max\{ 2 , \diam_d(X)^{-1} \}$ there exist two normal fuzzy sets $u,u^k \in \Fc(X)$ fulfilling that
		\[
		d_{E}(u,u^k) = d_{S}(u,u^k) = d_{0}(u,u^k) = \tfrac{1}{k}
		\]
		and
		\[
		d_{E}(\fuz{f}^n(u),\fuz{f}^n(u^k)) \leq d_{S}(\fuz{f}^n(u),\fuz{f}^n(u^k)) \leq d_{0}(\fuz{f}^n(u),\fuz{f}^n(u^k)) \leq \tfrac{1}{k} \quad \text{ for all } n \in \NN.
		\]
	\end{enumerate}
\end{lemma}
In the statement of Lemma~\ref{Lem:contraexpansive} we could have that $\diam_d(f(X))=\infty$ or that $\diam_d(X)=\infty$, respectively. In each of this cases, the value $\diam_d(f(X))^{-1}$ or $\diam_d(X)^{-1}$ would be taken equal to~$0$, so that the only limitation of the respective statement would be considering $k>2$.
\begin{proof}[Proof of Lemma~\ref{Lem:contraexpansive}]
	In both parts (a) and (b) of the statement we are going to fix some integer $k > 2$ and two distinct points $x,y \in X$, and then we will consider the normal fuzzy sets
	\begin{equation}\label{eq:u.u^k}
		u := \chi_{\{x\}}+\tfrac{1}{2}\chi_{\{y\}} \quad \text{ and } \quad u^{k} := \chi_{\{x\}}+\left(\tfrac{1}{2}-\tfrac{1}{k}\right)\chi_{\{y\}}.
	\end{equation}
	Moreover, for each integer $k>2$ we will consider the 2-piecewise-linear map $\xi_k:\II\longrightarrow\II$ defined as
	\begin{equation}\label{eq:xi_k}
		\xi_k(\alpha) :=
		\begin{cases}
			\tfrac{k}{k-2}\alpha & \text{ if } 0\leq \alpha\leq \tfrac{1}{2}-\tfrac{1}{k},\\[5pt]
			\tfrac{k}{k+2}\alpha + \tfrac{2}{k+2} & \text{ if } \tfrac{1}{2}-\tfrac{1}{k}< \alpha\leq 1,
		\end{cases}
	\end{equation}
	which fulfills that $\xi_k(0)=0$, that $\xi_k(\tfrac{1}{2}-\tfrac{1}{k})=\tfrac{1}{2}$ and that $\xi_k(1)=1$. It is not hard to check that
	\begin{equation}\label{eq:xi_k=1/k}
		\sup_{\alpha\in\II} |\xi_k(\alpha)-\alpha| = \left|\xi_k\left(\tfrac{1}{2}-\tfrac{1}{k}\right) - \left(\tfrac{1}{2}-\tfrac{1}{k}\right)\right| = \left|\tfrac{1}{2} - \left(\tfrac{1}{2}-\tfrac{1}{k}\right)\right| = \tfrac{1}{k},
	\end{equation}
	and that $\xi_k^{-1}(\tfrac{1}{2}) = \tfrac{1}{2}-\tfrac{1}{k}$. We are now ready to proceed with the proof.
	
	(a): Assume that $\diam_d(f(X))>0$. Then, fixed any integer $k > \max\{ 2 , \diam_d(f(X))^{-1} \}$ we can find points $x,y \in X$ such that $\tfrac{1}{k} \leq d(f(x),f(y)) \leq \diam_d(f(X))$. It follows that $f(x) \neq f(y)$, so that $x$ and $y$ are necessarily distinct and we can consider the sets $u,u^k \in \Fc(X)$ as defined in \eqref{eq:u.u^k}. Now, since $d_{E}(v,w) \leq d_{S}(v,w) \leq d_{0}(v,w)$ holds for every $v,w \in \Fc(X)$ as stated in Proposition~\ref{Pro:fuzzy.metrics}, to show that $u$ and $u^k$ fulfill the required conditions we only need to prove the next two inequalities:
	\begin{enumerate}[--]
		\item \textbf{Inequality 1}: \textit{We have that $d_{0}(\fuz{f}^n(u),\fuz{f}^n(u^k)) \leq \tfrac{1}{k}$ for $n \in \{0,1\}$}. Fix $n \in \{0,1\}$ and note that
		\[
		\fuz{f}^n(u) = \chi_{\{f^n(x)\}}+\tfrac{1}{2}\chi_{\{f^n(y)\}} = 
		\begin{cases}
			\chi_{\{x\}}+\tfrac{1}{2}\chi_{\{y\}} & \text{ if } n=0,\\[5pt]
			\chi_{\{f(x)\}}+\tfrac{1}{2}\chi_{\{f(y)\}} & \text{ if } n=1,
		\end{cases}
		\]
		but also that
		\[
		\fuz{f}^n(u^k) = \chi_{\{f^n(x)\}}+\left( \tfrac{1}{2}-\tfrac{1}{k} \right)\chi_{\{f^n(y)\}} = 
		\begin{cases}
			\chi_{\{x\}}+\left( \tfrac{1}{2}-\tfrac{1}{k} \right)\chi_{\{y\}} & \text{ if } n=0,\\[5pt]
			\chi_{\{f(x)\}}+\left( \tfrac{1}{2}-\tfrac{1}{k} \right)\chi_{\{f(y)\}} & \text{ if } n=1.
		\end{cases}
		\]
		Considering the strictly increasing homeomorphism $\xi_k:\II\longrightarrow\II$ as defined in \eqref{eq:xi_k} we have that
		\[
		\left[ \xi_k \circ \left( \fuz{f}^n(u^k) \right) \right]_{\alpha} = \left[ \fuz{f}^n(u^k) \right]_{\xi_k^{-1}(\alpha)} =
		\begin{cases}
			\{f^n(x),f^n(y)\} & \text{ if } 0\leq \alpha\leq \tfrac{1}{2},\\[5pt]
			\{f^n(x)\} & \text{ if } \tfrac{1}{2}<\alpha\leq 1.
		\end{cases}
		\]
		This implies that $\fuz{f}^n(u)=\xi_k \circ \fuz{f}^n(u^k)$ and hence $d_{\infty}(\fuz{f}^n(u),\xi_k \circ \fuz{f}^n(u^k))=0$. We finally deduce that
		\[
		d_{0}(\fuz{f}^n(u),\fuz{f}^n(u^k)) \leq \max\left\{ \sup_{\alpha\in\II} |\xi_k(\alpha)-\alpha| \ , \ d_{\infty}(\fuz{f}^n(u),\xi_k \circ \fuz{f}^n(u^k)) \right\} \overset{\eqref{eq:xi_k=1/k}}{=} \tfrac{1}{k}.
		\]
		
		\item \textbf{Inequality 2}: \textit{We have that $d_{E}(\fuz{f}(u),\fuz{f}(u^k)) \geq \tfrac{1}{k}$}. We use statement~(a) of Proposition~\ref{Pro:Hausdorff}: fixed any $0\leq\eps<\tfrac{1}{k}$ we will check that $\eend(\fuz{f}(u)) \not\subset \eend(\fuz{f}(u^k))+\eps$. In fact, $(f(y),\tfrac{1}{2}) \in \eend(\fuz{f}(u))$,
		\[
		\inf_{(z,\alpha) \in \eend(\fuz{f}(u^k))} \com{d}( (f(y),\tfrac{1}{2}) , (z,\alpha) ) = \inf_{(z,\alpha) \in \eend(\fuz{f}(u^k))} \max\{ d(f(y),z) , |\tfrac{1}{2}-\alpha| \},
		\]
		and for each $(z,\alpha) \in \eend(\fuz{f}(u^k))$ we have that
		\begin{align*}
			\max\{ d(f(y),z) , |\tfrac{1}{2}-\alpha| \} &= \left\{
			\begin{array}{ll}
				\max\{ d(f(y),z) , \tfrac{1}{2} \} & \text{ if } \alpha=0 \text{ and } z \in X, \hspace{1cm} \\[5pt]
				\max\{ d(f(y),f(y)) , |\tfrac{1}{2}-\alpha| \} & \text{ if } 0<\alpha\leq\tfrac{1}{2}-\tfrac{1}{k} \text{ and } z=f(y), \\[5pt]
				\max\{ d(f(y),f(x)) , |\tfrac{1}{2}-\alpha| \} & \text{ if } 0<\alpha\leq 1 \text{ and } z=f(x),
			\end{array}
			\right\} \\[7.5pt]
			&\geq \left\{
			\begin{array}{ll}
				\tfrac{1}{2} & \text{ if } \alpha=0 \text{ and } z \in X, \\[5pt]
				\tfrac{1}{k} & \text{ if } 0<\alpha\leq\tfrac{1}{2}-\tfrac{1}{k} \text{ and } z=f(y), \\[5pt]
				d(f(x),f(y)) & \text{ if } 0<\alpha\leq 1 \text{ and } z=f(x),
			\end{array}
			\right\} \geq \tfrac{1}{k} > \eps.
		\end{align*}
		We get that $d_{E}(\fuz{f}(u),\fuz{f}(u^k))>\eps$ and the arbitrariness of $0\leq\eps<\tfrac{1}{k}$ shows that $d_{E}(\fuz{f}(u),\fuz{f}(u^k))\geq\tfrac{1}{k}$.
	\end{enumerate}

	(b): Assume now that $\diam_d(X)>0$. Then, fixed any integer $k > \max\{ 2 , \diam_d(X)^{-1} \}$ we can find points $x,y \in X$ such that $\tfrac{1}{k} \leq d(x,y) \leq \diam_d(X)$. It follows that $x$ and $y$ are distinct and we can consider the sets $u,u^k \in \Fc(X)$ as defined in \eqref{eq:u.u^k}. Using again that $d_{E}(v,w) \leq d_{S}(v,w) \leq d_{0}(v,w)$ holds for every $v,w \in \Fc(X)$, we only need to prove the next two inequalities:
	\begin{enumerate}[--]
		\item \textbf{Inequality 1}: \textit{We have that $d_{E}(u,u^k) \geq \tfrac{1}{k}$}. As above, to check this inequality we use statement~(a) of Proposition~\ref{Pro:Hausdorff}: fixed any $0\leq\eps<\tfrac{1}{k}$ we will check that $\eend(u) \not\subset \eend(u^k)+\eps$. In fact, similarly to the previous case, the point $(y,\tfrac{1}{2}) \in \eend(u)$ fulfills that
		\[
		\inf_{(z,\alpha) \in \eend(u^k)} \com{d}( (y,\tfrac{1}{2}) , (z,\alpha) ) = \inf_{(z,\alpha) \in \eend(u^k)} \max\{ d(y,z) , |\tfrac{1}{2}-\alpha| \},
		\]
		and for each $(z,\alpha) \in \eend(u^k)$ we have that
		\begin{align*}
			\max\{ d(y,z) , |\tfrac{1}{2}-\alpha| \} &= \left\{
			\begin{array}{ll}
				\max\{ d(y,z) , \tfrac{1}{2} \} & \text{ if } \alpha=0 \text{ and } z \in X, \hspace{1cm} \\[5pt]
				\max\{ d(y,y) , |\tfrac{1}{2}-\alpha| \} & \text{ if } 0<\alpha\leq\tfrac{1}{2}-\tfrac{1}{k} \text{ and } z=y, \\[5pt]
				\max\{ d(y,x) , |\tfrac{1}{2}-\alpha| \} & \text{ if } 0<\alpha\leq 1 \text{ and } z=x,
			\end{array}
			\right\} \\[7.5pt]
			&\geq \left\{
			\begin{array}{ll}
				\tfrac{1}{2} & \text{ if } \alpha=0 \text{ and } z \in X, \\[5pt]
				\tfrac{1}{k} & \text{ if } 0<\alpha\leq\tfrac{1}{2}-\tfrac{1}{k} \text{ and } z=y, \\[5pt]
				d(x,y) & \text{ if } 0<\alpha\leq 1 \text{ and } z=x,
			\end{array}
			\right\} \geq \tfrac{1}{k} > \eps.
		\end{align*}
		We get that $d_E(u,u^k)>\eps$ and the arbitrariness of $0\leq\eps<\tfrac{1}{k}$ shows that $d_E(u,u^k)\geq\tfrac{1}{k}$.
		
		\item \textbf{Inequality 2}: \textit{We have that $d_{0}(\fuz{f}^n(u),\fuz{f}^n(u^k)) \leq \tfrac{1}{k}$ for all $n \in \NN_0$}. Fix any $n \in \NN_0$ and note that
		\[
		\fuz{f}^n(u) = 
		\begin{cases}
			\chi_{\{f^n(x)\}}+\tfrac{1}{2}\chi_{\{f^n(y)\}} & \text{ if } f^n(x)\neq f^n(y),\\[5pt]
			\chi_{\{f^n(x)\}} & \text{ if } f^n(x)=f^n(y),
		\end{cases}
		\]
		but also that
		\[
		\fuz{f}^n(u^k) =  
		\begin{cases}
			\chi_{\{f^n(x)\}}+\left( \tfrac{1}{2}-\tfrac{1}{k} \right)\chi_{\{f^n(y)\}} & \text{ if } f^n(x)\neq f^n(y),\\[5pt]
			\chi_{\{f^n(x)\}} & \text{ if } f^n(x)=f^n(y).
		\end{cases}
		\]
		If $f^n(x)=f^n(y)$ then we would have that $\fuz{f}^n(u)=\fuz{f}^n(u^k)$ and $d_{0}(\fuz{f}^n(u),\fuz{f}^n(u^k)) = 0 \leq \tfrac{1}{k}$, which would finish the proof in this case. Assume now that $f^n(x)\neq f^n(y)$. Considering then the strictly increasing homeomorphism $\xi_k:\II\longrightarrow\II$ as defined in \eqref{eq:xi_k} we have that
		\[
		\left[ \xi_k \circ \left( \fuz{f}^n(u^k) \right) \right]_{\alpha} = \left[ \fuz{f}^n(u^k) \right]_{\xi_k^{-1}(\alpha)} =
		\begin{cases}
			\{f^n(x),f^n(y)\} & \text{ if } 0\leq \alpha\leq \tfrac{1}{2},\\[5pt]
			\{f^n(x)\} & \text{ if } \tfrac{1}{2}<\alpha\leq 1.
		\end{cases}
		\]
		This implies that $\fuz{f}^n(u)=\xi_k \circ \fuz{f}(u^k)$ and hence $d_{\infty}(\fuz{f}^n(u),\xi_k \circ \fuz{f}^n(u^k))=0$. We finally deduce that
		\[
		d_{0}(\fuz{f}^n(u),\fuz{f}^n(u^k)) \leq \max\left\{ \sup_{\alpha\in\II} |\xi_k(\alpha)-\alpha| \ , \ d_{\infty}(\fuz{f}^n(u),\xi_k \circ \fuz{f}^n(u^k)) \right\} \overset{\eqref{eq:xi_k=1/k}}{=} \tfrac{1}{k}.\qedhere
		\]
	\end{enumerate}
\end{proof}

The proof of Lemma~\ref{Lem:contraexpansive} is now complete. Note that the choice of the fuzzy sets $u$ and $u^k$ considered in the previous proof admits several variations: for instance, adapting the requirement $k>2$, one could have taken $u=\chi_{\{x\}}+\alpha\chi_{\{y\}}$ for any $\alpha\in \ ]0,1[$ together with $u^k=\chi_{\{x\}}+(\alpha\pm\tfrac{1}{k})\chi_{\{y\}}$ instead than the sets $u$ and $u^k$ considered in \eqref{eq:u.u^k}. In any case, we are now ready to prove Theorem~\ref{The:contraexpansive}.

\begin{proof}[Proof of Theorem~\ref{The:contraexpansive}]
	(a): Assume first that the map $f:X\longrightarrow X$ is constant, so that there exists some point $x^* \in X$ fulfilling that $f(x)=x^*$ for all $x \in X$. In this case, it is not hard to check that
	\[
	\fuz{f}(u)=\chi_{\{x^*\}} \quad \text{ for all } u \in \Fc(X).
	\]
	The system $(\Fc(X),\fuz{f})$ is then contractive for $\lambda = 0$, regardless of the topology considered on $\Fc(X)$.
	
	Conversely, if the map $f:X\longrightarrow X$ is non-constant then we have that $\diam_d(f(X))>0$. Fix any metric $\rho \in \{d_{0},d_{S},d_{E}\}$ and, by contradiction, assume that the dynamical system $(\Fc(X),\fuz{f})$ is contractive when the space $\Fc(X)$ is endowed with the metric $\rho$. In that case there exists a real value $\lambda \in [0,1[$ such that $\rho(\fuz{f}(v),\fuz{f}(w)) \leq \lambda \rho(v,w)$ for every pair of sets $v,w \in \Fc(X)$. Applying statement~(a) of Lemma~\ref{Lem:contraexpansive} for any $k>\max\{ 2 , \diam_d(f(X))^{-1} \}$ we get $u,u^k \in \Fc(X)$ for which
	\[
	\tfrac{1}{k} = \rho(\fuz{f}(u),\fuz{f}(u^k)) \leq \lambda \rho(u,u^k) \leq \lambda \tfrac{1}{k} < \tfrac{1}{k},
	\]
	which is a contradiction.
	
	(b): Assume first that the set $X$ is a singleton. In this case, Lemma~\ref{Lem:isolated.points} implies that $\Fc(X)$ is also a singleton, so that $(\Fc(X),\fuz{f})$ reduces to a dynamical system consisting of a single fixed point.~It is straightforward to verify that such a system trivially satisfies the definitions of being expansive, expanding, and positively expansive, regardless of the topology considered on $\Fc(X)$.
	
	Conversely, if the set $X$ is not a singleton then we have that $\diam_d(X)>0$. As in part (a), let us fix any metric $\rho \in \{d_{0},d_{S},d_{E}\}$ and, by contradiction:
	\begin{enumerate}[--]
		\item Assume that the dynamical system $(\Fc(X),\fuz{f})$ is expansive when the space $\Fc(X)$ is endowed with the metric $\rho$. Then there exists a positive value $\lambda>1$ such that $\rho(\fuz{f}(v),\fuz{f}(w))\geq \lambda \rho(v,w)$ for every pair of points $v,w \in \Fc(X)$. Applying now statement (b) of Lemma~\ref{Lem:contraexpansive} for any $k>\max\{ 2 , \diam_d(X)^{-1} \}$ we get two fuzzy sets $u,u^k \in \Fc(X)$ fulfilling that
		\[
		\tfrac{1}{k} \geq \rho(\fuz{f}(u),\fuz{f}(u^k)) \geq \lambda \rho(u,u^k) = \lambda \tfrac{1}{k} > \tfrac{1}{k},
		\]
		which is a contradiction.
		
		\item Assume that the dynamical system $(\Fc(X),\fuz{f})$ is expanding when the space $\Fc(X)$ is endowed with the metric $\rho$. Then there exist $\eps>0$ and $\lambda>1$ such that $\rho(\fuz{f}(v),\fuz{f}(w)) > \lambda \rho(v,w)$ for every pair of normal fuzzy sets $v,w \in \Fc(X)$ fulfilling that $0<\rho(v,w)<\eps$. Applying now statement~(b) of Lemma~\ref{Lem:contraexpansive} for any $k>\max\{ 2 , \diam_d(X)^{-1} , \tfrac{1}{\eps} \}$ we get two normal fuzzy sets $u,u^k \in \Fc(X)$ fulfilling that $0<\rho(u,u^k)=\tfrac{1}{k}<\eps$ and also that
		\[
		\tfrac{1}{k} \geq \rho(\fuz{f}(u),\fuz{f}(u^k)) > \lambda \rho(u,u^k) = \lambda \tfrac{1}{k} > \tfrac{1}{k},
		\]
		which is a contradiction.
		
		\item Assume that the system $(\Fc(X),\fuz{f})$ is positively expansive when the space of fuzzy sets $\Fc(X)$ is endowed with the metric $\rho$. Then there exists a constant $\delta>0$ such that for any pair of distinct sets $v \neq w$ in $\Fc(X)$ there is some $n \in \NN_0$ fulfilling that $\rho(\fuz{f}^n(v),\fuz{f}^n(w)) > \delta$. Applying once more statement (b) of Lemma~\ref{Lem:contraexpansive} for any $k > \max\{ 2 , \diam_d(X)^{-1} , \tfrac{1}{\delta} \}$, we finally get two distinct fuzzy sets $u\neq u^k$ in $\Fc(X)$ fulfilling that
		\[
		\delta > \tfrac{1}{k} \geq \rho(\fuz{f}^n(u),\fuz{f}^n(u^k)) \quad \text{ for every } n \in \NN_0,
		\]
		which is a contradiction.\qedhere
	\end{enumerate}
\end{proof}

The proof of Theorem~\ref{The:contraexpansive} is now complete. To finish Section~\ref{Sec_3:contraexpansive} we will solve \cite[Problem~3.7]{JardonSan2021_FSS_expansive} in the negative. This is the only problem, from \cite[Problems~3.6,~3.7~and~4.3]{JardonSan2021_FSS_expansive} and \cite[Problems~5.8~and~5.12]{JardonSanSan2020_FSS_some}, that cannot be directly solved via Theorem~\ref{The:contraexpansive} (see Remark~\ref{Rem:solved} for more details). 

\subsection{The solution to an expanding problem and a useful lemma}

In \cite{JardonSan2021_FSS_expansive} the authors also left \cite[Problem~3.7]{JardonSan2021_FSS_expansive} as an open question: {\em Let $f$ be a continuous map from a metric space $(X,d)$ into itself and assume that $(\Kc(X),\com{f})$ is expanding. Is it true that there exists~$\eps>0$ such that $d_{0}(\fuz{f}(u),\fuz{f}(v)) \geq d_{0}(u,v)$ for all $u,v \in \Fc(X)$ fulfilling that $d_{0}(u,v)<\eps$?} We will solve this problem in the negative with a counterexample. We also use such example of dynamical system to show that the inequality ``$d_{0}(u,u^k)\leq\tfrac{1}{k}$'' in statement (b) of Lemma~\ref{Lem:contraexpansive} can be strict:

\begin{example}\label{Exa:expanding}
	\textit{There exists a system $(X,f)$ such that $(\Kc(X),\com{f})$ is expanding while for each $\eps>0$ there exists $u^{\eps},v^{\eps} \in \Fc(X)$ fulfilling that $\rho(\fuz{f}(u^{\eps}),\fuz{f}(v^{\eps})) < \rho(u^{\eps},v^{\eps}) < \eps$ for all $\rho \in \{ d_{0} , d_{S} , d_{E} \}$.}
	\begin{proof}
		Let $X=\{a,b\}$ with the usual discrete metric, i.e.\ given $x,y \in X$ we have that $d(x,y)=0$ when $x=y$ and $d(x,y)=1$ when $x \neq y$. Let $f:X\longrightarrow X$ with $f(a)=a$ and $f(b)=a$. Since both metric spaces $(X,d)$ and $(\Kc(X),d_H)$ are discrete, it is not hard to see that $(X,f)$ and $(\Kc(X),\com{f})$ are expanding dynamical systems by picking any positive value $0<\eps\leq 1$ and any $\lambda>1$. Now, let $\eps>0$ and fix an integer $k > \max\{ 2 , \tfrac{1}{\eps} \}$. Similar to what we did in \eqref{eq:u.u^k}, we can consider the fuzzy sets~$u^{\eps}:=\chi_{\{a\}}+\tfrac{1}{2}\chi_{\{b\}}$ and $v^{\eps}:=\chi_{\{a\}}+\left(\tfrac{1}{2}-\tfrac{1}{k}\right)\chi_{\{b\}}$. Since $\diam_d(X)=1$ and $k>\max\{ 2 , \diam_d(X)^{-1} \}$, the proof of part~(b) of Lemma~\ref{Lem:contraexpansive} shows that $\rho(u^{\eps},v^{\eps})=\tfrac{1}{k}<\eps$. Noticing that
		\[
		\fuz{f}(u^{\eps}) = \max\{\chi_{\{f(a)\}},\tfrac{1}{2}\chi_{\{f(b)\}}\} = \chi_{\{a\}} = \max\{\chi_{\{f(a)\}},\left(\tfrac{1}{2}-\tfrac{1}{k}\right)\chi_{\{f(b)\}}\} = \fuz{f}(v^{\eps}),
		\]
		we have that $\rho(\fuz{f}(u^{\eps}),\fuz{f}(v^{\eps})) = 0 < \rho(u^{\eps},v^{\eps}) = \tfrac{1}{k} < \eps$, finishing the counterexample.
	\end{proof}
\end{example}

The system $(X,f)$ from Example~\ref{Exa:expanding} admits several variations; for instance, its disjoint union with any expansive system is again a counterexample to \cite[Problem~3.7]{JardonSan2021_FSS_expansive} for which the expanding property of $(\Kc(X),\com{f})$ is not necessarily trivial (the details are left to the reader). Following \cite{JardonSan2021_FSS_expansive,WuZhangChen2020_FSS_answers} we will end this section proving that, although Theorem~\ref{The:contraexpansive} prevents $(\Fc_{0}(X),\fuz{f})$, $(\Fc_{S}(X),\fuz{f})$ and $(\Fc_{E}(X),\fuz{f})$ from being contractive or expansive in general, these systems may still behave in an ``almost'' contractive or expansive manner whenever $(X,f)$ itself has such a property.

\begin{lemma}\label{Lem:almost.contraexpansive}
	For a continuous map $f:X\longrightarrow X$ on a metric space $(X,d)$ and any $\rho \in \{ d_{0} , d_{S} , d_{E} \}$:
	\begin{enumerate}[{\em(a)}]
		\item If $(X,f)$ is a contractive dynamical system, then $\rho(\fuz{f}(u),\fuz{f}(v)) \leq \rho(u,v)$ for all $u,v \in \Fc(X)$.
		
		\item If $(X,f)$ is an expansive dynamical system, then $\rho(\fuz{f}(u),\fuz{f}(v)) \geq \rho(u,v)$ for all $u,v \in \Fc(X)$.
	\end{enumerate}
\end{lemma}
\begin{proof}
	(a): The case $\rho=d_{0}$ was proved in \cite[Lemma~6]{WuZhangChen2020_FSS_answers}. Thus, let $u,v \in \Fc(X)$ and note that, by statement~(a) of Proposition~\ref{Pro:Hausdorff}, the definition of~$d_{S}$~and~$d_{E}$, and their symmetry, it suffices to show that $\send(\fuz{f}(u)) \subset \send(\fuz{f}(v))+\eps$ or that $\eend(\fuz{f}(u)) \subset \eend(\fuz{f}(v))+\eps$ whenever $\send(u) \subset \send(v)+\eps$ or $\eend(u) \subset \eend(v) + \eps$ for some $\eps>0$. We show the endograph case, the sendograph one being simpler, i.e.\ assume that $\eend(u) \subset \eend(v)+\eps$. Let $(z,\alpha) \in \eend(\fuz{f}(u)) \cap ([\fuz{f}(u)]_0\times\II) = \send(\fuz{f}(u))$, so that $z \in [\fuz{f}(u)]_{\alpha}$. Then there exist $x \in u_{\alpha}$ with $f(x)=z$, so that $(x,\alpha) \in \eend(u)$ and by assumption there is some $(y,\beta) \in \eend(v)$ with $\max\{ d(x,y) , |\alpha-\beta| \} = \com{d}( (x,\alpha) , (y,\beta) ) \leq \eps$. Since $(X,f)$ is contractive,
	\[
	\com{d}( (z,\alpha) , (f(y),\beta) ) = \max\{ d(f(x),f(y)) , |\alpha-\beta| \} \leq \max\{ d(x,y) , |\alpha-\beta| \} \leq \eps,
	\]
	so that $(z,\alpha) \in \{(f(y),\beta)\}+\eps \subset \eend(\fuz{f}(v))+\eps$. For the case $(z,\alpha) \in \eend(\fuz{f}(u))\setminus([\fuz{f}(u)]_0\times\II)$ we have that $\alpha=0$ and hence that $(z,\alpha) \in \eend(\fuz{f}(v)) \subset \eend(\fuz{f}(v))+\eps$, finishing the proof.
	
	(b): This was proved in \cite[Propositions~3.3~and~3.4]{JardonSan2021_FSS_expansive}.
\end{proof}

We end Section~\ref{Sec_3:contraexpansive} here, but contractive dynamical systems will be revisited in Subsection~\ref{SubSec_5.2:contractions} in connection with the shadowing property, and Lemma~\ref{Lem:almost.contraexpansive} will be crucial (see Theorem~\ref{The:contractions}).

\section{Chain recurrence, chain transitivity and chain mixing}\label{Sec_4:chains}

In this section we study the notions of chain recurrence, chain transitivity and chain mixing. We begin by proving Theorem~\ref{The:chains}, which asserts that each of these properties is presented by $(\Kc(X),\com{f})$ if and only if $(\Fc_{\infty}(X),\fuz{f})$, $(\Fc_{0}(X),\fuz{f})$ and $(\Fc_{S}(X),\fuz{f})$ do as well. These equivalences are, to the best of our knowledge, completely new results. We then show that the system $(\Fc_{E}(X),\fuz{f})$ is chain recurrent, chain transitive and chain mixing if and only if the underlying map $f:(X,d)\longrightarrow(X,d)$ has dense range (see Theorem~\ref{The:chains.E}). This exhibits, again, an extreme behaviour for the endograph metric $d_{E}$.

Along this section, given any system $(X,f)$ on a metric space $(X,d)$ and any positive integer $N \in \NN$, we will denote by $f_{(N)}:X^N\longrightarrow X^N$ the {\em $N$-fold direct product} of $f$ with itself. That is, the pair denoted by $(X^N,f_{(N)})$ will be the dynamical system
\[
f_{(N)} := \underbrace{f\times\cdots\times f}_{N} : \underbrace{X\times\cdots\times X}_{N} \longrightarrow \underbrace{X\times\cdots\times X}_{N},
\]
where $X^N:=X\times\cdots\times X$ is the {\em $N$-fold direct product} of $X$, and where the respective {\em $N$-fold direct map} is defined as $f_{(N)}\left((x_1,...,x_N)\right) := (f(x_1),...,f(x_N))$ for each $N$-tuple $(x_1,...,x_N) \in X^N$. In addition, the space $X^N$ will be endowed with the metric $d_{(N)}:X^N\times X^N\longrightarrow[0,\infty[$ defined as
\[
d_{(N)}\left( (x_1,...,x_N) , (y_1,...,y_N) \right) := \max_{1\leq l\leq N} d(x_l,y_l) \quad \text{ for each pair } (x_1,...,x_N) , (y_1,...,y_N) \in X^N.
\]

\subsection{Chain-type properties for the supremum, Skorokhod and sendograph metrics}\label{SubSec_4.1:chains}

Let $f:X\longrightarrow X$ be a continuous map acting on a metric space $(X,d)$. Then:
\begin{enumerate}[--]
	\item given a positive value $\delta>0$ and a positive integer $n \in \NN$, we say that a finite sequence $(x_j)_{j=0}^n$ of points in $X$ is a {\em $d$-$\delta$-chain} (from $x_0$ to $x_n$, of length $n$, and for the map $f$) if we have that
	\[
	d(f(x_j),x_{j+1})<\delta \quad \text{ for all } 0\leq j\leq n-1;
	\]

	\item a point $x \in X$ is called {\em chain recurrent} for the map $f$ if for each positive value $\delta>0$ there exists a $d$-$\delta$-chain from $x$ to itself for the map $f$; and the system $(X,f)$ is said to be {\em chain recurrent} if every point of the underlying space $X$ is chain recurrent for the map $f$;
	
	\item the system $(X,f)$ is called {\em chain transitive} if for each pair of points $x,y \in X$ and each $\delta>0$ there exists a $d$-$\delta$-chain from $x$ to $y$ for the map $f$;
	
	\item the system $(X,f)$ is called {\em chain weakly-mixing} if the product system $(X^2,f_{(2)})$ is chain transitive;
	
	\item and the system $(X,f)$ is called {\em chain mixing} if for each pair of points $x,y \in X$ and each $\delta>0$ there is a positive integer $n_0 \in \NN$ such that for every integer $n \geq n_0$ there exists a $d$-$\delta$-chain from $x$ to $y$ of length $n$ for the map $f$.
\end{enumerate}
From now on we will usually omit the words ``\textit{for the map $f$}''. These properties generalize the notions of {\em topological recurrence} (or {\em non-wandering}, see \cite[Remark~2.1]{AlvarezLoPe2025_FSS_recurrence}), {\em topological transitivity}, {\em topological weak-mixing} and {\em topological mixing} respectively (see \cite[Section~2]{RichersonWise2008_TA_chain}). The ``\textit{$\delta$-chains}'' were introduced by Conley \cite{Conley1988_ETDS_the}, and then further developed by Conley~\cite{Conley1978_book_isolated} and Bowen~\cite{Bowen1975_JDE_omega-limit}, together with the notion of {\em pseudo-trajectory} (see Section~\ref{Sec_5:shadowing} below). Such concepts have had a great impact in the qualitative theory of Dynamical Systems, and they have recently been studied in Linear Dynamics (see \cite{BernardesPe2025_AdvM_on,LopezPa2025_JDE_shifts}).

In the context of the compact and fuzzy extensions $(\Kc(X),\com{f})$ and $(\Fc(X),\fuz{f})$, these chain-type properties have been considered in \cite{AyalaDaRo2018_arXiv_dynamics,FernandezGoodPulRa2015_CSF_chain,KhanKu2013_FEJDS_recurrence}, albeit under additional assumptions. In all three works, the underlying metric space $(X,d)$ is assumed to be compact, a condition that is repeatedly used in their arguments. Moreover, in \cite{KhanKu2013_FEJDS_recurrence}, which is the only one of these works in which $\fuz{f}$ is considered, the dynamics of this map are studied in a space larger than that of normal fuzzy sets: they do not require their fuzzy sets to attain the value $1$ and, consequently, their system cannot be chain transitive nor chain mixing. As we work in general, not necessarily compact metric spaces, we will briefly reprove the implications established in \cite{AyalaDaRo2018_arXiv_dynamics,FernandezGoodPulRa2015_CSF_chain,KhanKu2013_FEJDS_recurrence} in the following theorem, which is the first main result of this section. To the best of our knowledge, the fuzzy part of our next result is completely new.

\begin{theorem}\label{The:chains}
	Let $f:X\longrightarrow X$ be a continuous map on a metric space $(X,d)$. Then, the following statements are equivalent:
	\begin{enumerate}[{\em(i)}]
		\item $(X,f)$ is chain recurrent (resp.\ chain weakly-mixing or chain mixing);
		
		\item $(X^N,f_{(N)})$ is chain recurrent (resp.\ chain transitive or chain mixing) for all $N \in \NN$;
		
		\item $(\Kc(X),\com{f})$ is chain recurrent (resp.\ chain transitive or chain mixing);
		
		\item $(\Fc_{\infty}(X),\fuz{f})$ is chain recurrent (resp.\ chain transitive or chain mixing);
		
		\item $(\Fc_{0}(X),\fuz{f})$ is chain recurrent (resp.\ chain transitive or chain mixing);
		
		\item $(\Fc_{S}(X),\fuz{f})$ is chain recurrent (resp.\ chain transitive or chain mixing).
	\end{enumerate}
\end{theorem}
Before proving Theorem~\ref{The:chains} let us mention that condition (i) cannot be relaxed for the case of chain transitivity. In fact, let $X=\{a,b\}$ with the discrete metric described in Example~\ref{Exa:expanding}, and let $f(a)=b$ and $f(b)=a$. Thus, the dynamical system $(X,f)$ is chain transitive, but for the map $\com{f}$ there is no $d_H$-$\delta$-chain in $\Kc(X)$ from $\{a,b\}$ to $\{a\}$ whenever $0<\delta<1$. An example without isolated points is given in \cite[Example~2]{RichersonWise2008_TA_chain}, but it is, in spirit, completely analogous to the example exhibited here.
\begin{proof}[Proof of Theorem~\ref{The:chains}]
	(i) $\Rightarrow$ (ii): Assume first that $(X,f)$ is chain recurrent, let $N \in \NN$, fix any point $(x_1,x_2,...,x_N) \in X^N$ and consider any $\delta>0$. Thus, for each $1\leq l\leq N$ there exists a $d$-$\delta$-chain from $x_l$ to itself and of length $n_l \in \NN$, denoted by $(x_j^l)_{j=0}^{n_l}$. Letting $n := \prod_{j=1}^{N} n_j$ and concatenating the chain $(x_j^l)_{j=0}^{n_l}$ with itself $n/n_l$-times we get, for each $1 \leq l \leq N$, a $d$-$\delta$-chain $(y_j^l)_{j=0}^n$ from $x_l$ to itself and of length $n$. Thus, the sequence $((y^1_j,y^2_j,...,y^N_j))_{j=0}^n$ is a $d_{(N)}$-$\delta$-chain from $(x_1,x_2,...,x_N)$ to itself.
	
	Let now $(X,f)$ be chain weakly-mixing. It easily follows that $(X,f)$ itself is chain transitive, so that the map $f:(X,d)\longrightarrow(X,d)$ has dense range. Now we will complete the proof by induction: assume that $(X^N,f_{(N)})$ is chain transitive for some integer $N \geq 2$ and let us prove that in this case the system $(X^{N+1},f_{(N+1)})$ is also chain transitive. In fact, given two points $(x_1,x_2,...,x_N,x_{N+1})$ and $(y_1,y_2,...,y_N,y_{N+1})$ in $X^{N+1}$ and any $\delta>0$ we can use that $(X^N,f_{(N)})$ is chain transitive to find two $d_{(N)}$-$\delta$-chains: the first one $((x^1_j,x^2_j,...,x^N_j))_{j=0}^{n_1}$ from $(x_1,x_2,x_3,...,x_N)$ to $(x_1,x_1,x_3,...,x_N)$, and the second one $((y^1_j,y^2_j,...,y^N_j))_{j=0}^{n_2}$ from $(y_1,y_1,y_3,...,y_N)$ to $(y_1,y_2,y_3,...,y_N)$. Using that $f^{n_2}$ has dense range because $f$ does, we can then choose a point $z \in f^{-n_2}(\Bc_d(y_{N+1},\delta))$. Again, since $(X^N,f_{(N)})$ is chain transitive we can find a $d_{(N)}$-$\delta$-chain $((z^1_j,z^2_j,...,z^N_j))_{j=0}^{n_3}$ from $(x_1,x_3,...,x_N,f^{n_1}(x_{N+1}))$ to $(y_1,y_3,...,y_N,z)$. Concatenating the following finite sequences of $(N+1)$-tuples
	\[
	((x^1_j,x^2_j,...,x^N_j,f^j(x_{N+1})))_{j=0}^{n_1}, \quad ((z^1_j,z^1_j,z^2_j,...,z^N_j))_{j=0}^{n_3}, \quad \text{ and } ((y^1_j,y^2_j,...,y^N_j,f^j(z)))_{j=0}^{n_2},
	\]
	we get a $d_{(N+1)}$-$\delta$-chain from $(x_1,x_2,...,x_N,x_{N+1})$ to $(y_1,y_2,...,y_N,f^{n_2}(z))$. Since $d(f^{n_2}(z),y_{N+1})<\delta$, replacing $f^{n_2}(z)$ with $y_{N+1}$ we get a $d_{(N+1)}$-$\delta$-chain from $(x_1,x_2,...,x_N,x_{N+1})$ to $(y_1,y_2,...,y_N,y_{N+1})$.
	
	Finally, assume that $(X,f)$ is chain mixing, let $N \in \NN$, fix any pair of points $(x_1,x_2,...,x_N)$ and~$(y_1,y_2,...,y_N)$ in $X^N$ and consider any $\delta>0$. Using the chain mixing condition on $x_l$ and $y_l$ for every $1 \leq l \leq N$ we can find some positive integer $n_0 \in \NN$ such that, for any $n \geq n_0$ and any $1 \leq l \leq N$, there exists a $d$-$\delta$-chain $(x^l_j)_{j=0}^n$ from $x_l$ to $y_l$ of length $n$. Hence, the sequence $((x^1_j,x^2_j,...,x^N_j))_{j=0}^n$ is a $d_{(N)}$-$\delta$-chain from $(x_1,x_2,...,x_N)$ to $(y_1,y_2,...,y_N)$ of length $n \geq n_0$, finishing the proof.
	
	(ii) $\Rightarrow$ (iii): For this implication we will use that \textit{the finite subsets of $X$ are dense in $(\Kc(X),d_H)$}, which can be trivially deduced from the definition of the Vietoris topology (see Subsection~\ref{SubSec_2.1:extensions}). Thus, given a pair of (not necessarily distinct) compact sets $K,L \in \Kc(X)$ and any $\delta>0$, using the previous density fact together with the $d_H$-continuity of $\com{f}$ we can find two (not necessarily distinct) finite sets $K'=\{x_1,x_2,...,x_{N_K}\} \subset X$ and $L'=\{y_1,y_2,...,y_{N_L}\} \subset X$, with cardinals $N_K,N_L \in \NN$, and such that
	\[
	d_H(K,K')<\tfrac{\delta}{2}, \quad d_H(\com{f}(K),\com{f}(K'))<\tfrac{\delta}{2} \quad \text{ and } d_H(L',L)<\tfrac{\delta}{2}.
	\]
	Thus, setting $N := \max\{ N_K , N_L \}$ and repeating points if necessary, i.e.\ when $N_K<N_L$ or $N_L<N_K$, we can construct two $N$-tuples $(x_1,x_2,...,x_N)$ and $(y_1,y_2,...,y_N)$ in the product space $X^N$ fulfilling that $K'=\{x_1,x_2,...,x_N\}$ and $L'=\{y_1,y_2,...,y_N\}$. These $N$-tuples can be taken equal when $K=L$. Note that, given any $d_{(N)}$-$\tfrac{\delta}{2}$-chain $((x^1_j,x^2_j,...,x^N_j))_{j=0}^n$ from $(x_1,x_2,...,x_N)$ to $(y_1,y_2,...,y_N)$, letting
	\[
	K_0 := K, \quad K_n := L, \quad \text{ and } K_j := \{ x^1_j , x^2_j , ... , x^N_j \} \quad \text{ for each } 1 \leq j \leq n-1,
	\]
	it can be checked that $(K_j)_{j=0}^n$ is a $d_H$-$\delta$-chain from $K$ to $L$. Hence, if the system~$(X^N,f_{(N)})$ is chain recurrent (resp.\ chain transitive or chain mixing) for all $N \in \NN$, then one can easily deduce that so is the extension $(\Kc(X),\com{f})$ by considering adequate $d_{(N)}$-$\tfrac{\delta}{2}$-chains as exposed in the previous reasoning.
	
	\noindent\textit{The implication {\em(ii)} $\Rightarrow$ {\em(iii)} is now complete. However, below we will use that if given any $N \in \NN$ we consider $N$ pairs of (not necessarily distinct) sets $K^1,K^2,...,K^N \in \Kc(X)$ and $L^1,L^2,...,L^N \in \Kc(X)$, then a completely similar proof shows that {\em(ii)} implies the following strengthened version of {\em(iii)}:}
	\begin{enumerate}
		\item[(iii')] {\em $(\Kc(X)^N,\com{f}_{(N)})$ is chain recurrent (resp.\ chain transitive or chain mixing) for all $N \in \NN$.}
	\end{enumerate}
	
	(iii) $\Rightarrow$ (i): Assume first that $(\Kc(X),\com{f})$ is chain recurrent, fix any point $x \in X$ and let $\delta>0$. Thus, there exists a $d_H$-$\delta$-chain $(K_j)_{j=0}^n$ from $\{x\}$ to itself. Let $x_0 := x \in \{x\} = K_0$, and from there pick points $x_j \in K_j$ for each $1 \leq j\leq n$ fulfilling that $d(f(x_j),x_{j+1})<\delta$ for all $0 \leq j \leq n-1$. It is not hard to check that $(x_j)_{j=0}^n$ is a $d$-$\delta$-chain from $x$ to itself, so that $(X,f)$ is chain recurrent.
	
	Let now $(\Kc(X),\com{f})$ be chain transitive, fix points $(x_1,x_2),(y_1,y_2) \in X^2$ and let $\delta>0$. Thus, picking any point $z \in X$, there exist two $d_H$-$\delta$-chains: the first one $(K_j)_{j=0}^{n_1}$ from $\{x_1,x_2\}$ to $\{z\}$, and the second one $(L_j)_{j=0}^{n_2}$ from $\{z\}$ to $\{y_1,y_2\}$. Let $x^1_0 := x_1$, $x^2_0 := x_2$, $y^1_{n_2} := y_1$ and $y^2_{n_2} := y_2$. Hence, given any $l \in \{1,2\}$ we can pick points $x^l_j \in K_j$ for each $1 \leq j \leq n_1$ fulfilling that $d(f(x^l_j),x^l_{j+1})<\delta$ for all $0 \leq j \leq n_1-1$, but we can also pick points $y^l_j \in L_j$ for each $0 \leq j \leq n_2-1$ fulfilling that $d(f(y^l_j),y^l_{j+1})<\delta$ for all $0 \leq j \leq n_2-1$. Concatenating the sequences $((x^1_j,x^2_j))_{j=0}^{n_1}$ and $((y^1_j,y^2_j))_{j=0}^{n_2}$ we get a $d_{(2)}$-$\delta$-chain from $(x_1,x_2)$ to $(y_1,y_2)$, so that $(X,f)$ is chain weakly-mixing.
	
	Finally, assume that $(\Kc(X),\com{f})$ is chain mixing, fix any pair of points $x,y \in X$ and let $\delta>0$. Using the chain mixing condition on $\{x\}$ and $\{y\}$ we can find a positive integer $n_0 \in \NN$ such that for any $n \geq n_0$ there exists a $d_H$-$\delta$-chain $(K_j)_{j=0}^n$ from $\{x\}$ to $\{y\}$. Letting $x_0 := x \in \{x\} = K_0$, and from there picking points $x_j \in K_j$ for each $1 \leq j\leq n$ fulfilling that $d(f(x_j),x_{j+1})<\delta$ for all $0 \leq j \leq n-1$,  it is not hard to check that $(x_j)_{j=0}^n$ is a $d$-$\delta$-chain from $x$ to $y$ of length $n$. Thus, $(X,f)$ is chain mixing.
	
	\noindent\textit{The equivalences {\em(i)} $\Leftrightarrow$ {\em(ii)} $\Leftrightarrow$ {\em(iii)} are now complete. However, as advanced above, in the following implication we will use that {\em(i)} $\Leftrightarrow$ {\em(ii)} $\Leftrightarrow$ {\em(iii)} $\Leftrightarrow$ {\em(iii')}, i.e.\ we use {\em(iii')}, the strengthened version of statement {\em(iii)} that we have included after proving the implication {\em(ii)} $\Rightarrow$ {\em(iii)}.}
	
	(iii') $\Rightarrow$ (iv): Given a pair of (not necessarily distinct) fuzzy sets $u,v \in \Fc(X)$ and any $\delta>0$, using the $d_{\infty}$-continuity of the map $\fuz{f}$ we can find some positive value $0 < \eps \leq \tfrac{\delta}{2}$ fulfilling that
	\begin{equation}\label{eq:fuz{f}(B_infty(u,eps))}
		\fuz{f}\left( \Bc_{\infty}(u,\eps) \right) \subset \Bc_{\infty}(\fuz{f}(u),\tfrac{\delta}{2}).
	\end{equation}
	Applying Lemma~\ref{Lem:eps.pisos}, two times when $u \neq v$, we can find values $0 = \alpha_0 < \alpha_1 < \alpha_2 < ... < \alpha_N = 1$ such that $\max\{ d_H(u_{\alpha}, u_{\alpha_{l+1}}) , d_H(v_{\alpha}, v_{\alpha_{l+1}}) \} < \eps$ for all $\alpha \in \ ]\alpha_l, \alpha_{l+1}]$ and $0\leq l\leq N-1$, but also fulfilling  that $\max\{ d_H(u_0,u_{\alpha_1}) , d_H(v_0,v_{\alpha_1}) \} < \eps$. Considering the fuzzy sets
	\[
	u' := \max_{1\leq l\leq N} \left( \alpha_l \cdot \chi_{u_{\alpha_l}} \right) \quad \text{ and } \quad v' := \max_{1\leq l\leq N} \left( \alpha_l \cdot \chi_{v_{\alpha_l}} \right),
	\]
	we can check that $\max\{ d_{\infty}(u,u') , d_{\infty}(v',v) \}<\eps$. In fact, since $u'_0=u_{\alpha_1}$ and $u'_{\alpha} = u_{\alpha_{l+1}}$ for every $\alpha \in \ ]\alpha_l,\alpha_{l+1}]$ with $0 \leq l \leq N-1$, we have that
	\[
	d_{\infty}(u,u') = \sup_{\alpha\in\II} d_H(u_{\alpha},u'_{\alpha}) = \left\{
	\begin{array}{lcc}
		d_H\left( u_{\alpha} , u_{\alpha_1} \right), & \text{ if } \alpha \in [0,\alpha_1] \hspace{3.1cm} \\[5pt]
		d_H\left( u_{\alpha} ,u_{\alpha_{l+1}} \right), & \text{ if } \alpha \in \ ]\alpha_l,\alpha_{l+1}] \, , 1\leq l\leq N-1
	\end{array}
	\right\} < \eps,
	\]
	and the same happens for $d_{\infty}(v',v)$. We claim that, given any $(d_H)_{(N)}$-$\tfrac{\delta}{2}$-chain $((K^1_j,K^2_j,...,K^N_j))_{j=0}^n$ from $(u_{\alpha_1},u_{\alpha_2},...,u_{\alpha_N})$ to $(v_{\alpha_1},v_{\alpha_2},...,v_{\alpha_N})$ in $\Kc(X)^N$, if we set
	\[
	u^0 := u, \quad u^n := v, \quad \text{ and } u^j := \max_{1\leq l\leq N} \left( \alpha_l \cdot \chi_{K^l_j} \right) \quad \text{ for each } 1 \leq j \leq n-1,
	\]
	then the sequence $(u^j)_{j=0}^n$ is a $d_{\infty}$-$\delta$-chain from $u$ to $v$. Indeed, using \eqref{eq:fuz{f}(B_infty(u,eps))} together with statement (b) of Proposition~\ref{Pro:Hausdorff} and the fact that $((K^1_j,K^2_j,...,K^N_j))_{j=0}^n$ is a $(d_H)_{(N)}$-$\tfrac{\delta}{2}$-chain we obtain that
	\begin{align*}
		d_{\infty}(\fuz{f}(u^0),u^1) &= d_{\infty}(\fuz{f}(u),u^1) \leq d_{\infty}(\fuz{f}(u),\fuz{f}(u')) + d_{\infty}(\fuz{f}(u'),u^1) \overset{\eqref{eq:fuz{f}(B_infty(u,eps))}}{<} \tfrac{\delta}{2} + \sup_{\alpha\in\II} d_H(\com{f}(u'_{\alpha}),u^1_{\alpha}) \\
		&= \tfrac{\delta}{2} + \max_{1\leq k\leq N} d_H\left( \bigcup_{k \leq l \leq N} \com{f}(u_{\alpha_l}), \bigcup_{k\leq l \leq N} K^l_1 \right) \leq \tfrac{\delta}{2} + \max_{1\leq l\leq N} d_H(\com{f}(u_{\alpha_l}),K^l_1) < \delta.
	\end{align*}
	In addition, given any $1 \leq j \leq n-2$, using again statement (b) of Proposition~\ref{Pro:Hausdorff} and the fact that $((K^1_j,K^2_j,...,K^N_j))_{j=0}^n$ is a $(d_H)_{(N)}$-$\tfrac{\delta}{2}$-chain we have that
	\begin{align*}
		d_{\infty}(\fuz{f}(u^j),u^{j+1}) &= \sup_{\alpha\in\II} d_H(\com{f}(u^j_{\alpha}),u^{j+1}_{\alpha}) = \max_{1\leq k\leq N} d_H\left( \bigcup_{k\leq l \leq N} \com{f}(K^l_j), \bigcup_{k\leq l \leq N} K^l_{j+1} \right) \\
		&\leq \max_{1\leq l\leq N} d_H(\com{f}(K^l_j),K^l_{j+1}) < \tfrac{\delta}{2} < \delta.
	\end{align*}
	And finally, using once more statement (b) of Proposition~\ref{Pro:Hausdorff} and the fact that $((K^1_j,K^2_j,...,K^N_j))_{j=0}^n$ is a $(d_H)_{(N)}$-$\tfrac{\delta}{2}$-chain we get that
	\begin{align*}
		d_{\infty}(\fuz{f}(u^{n-1}),u^n) &= d_{\infty}(\fuz{f}(u^{n-1}),v) \leq d_{\infty}(\fuz{f}(u^{n-1}),v') + d_{\infty}(v',v) < \sup_{\alpha\in\II} d_H(\com{f}(u^{n-1}_{\alpha}),v'_{\alpha}) + \eps\\
		&\leq \max_{1\leq k\leq N} d_H\left( \bigcup_{k\leq l \leq N} \com{f}(K^l_{n-1}) , \bigcup_{k \leq l \leq N} v_{\alpha_l} \right) + \tfrac{\delta}{2} \leq \max_{1\leq l\leq N} d_H(\com{f}(K^l_{n-1}),v_{\alpha_l}) + \tfrac{\delta}{2} < \delta.
	\end{align*}
	Hence, if $(\Kc(X)^N,\com{f}_{(N)})$ is chain recurrent (resp.\ chain transitive or chain mixing) for all $N \in \NN$, considering adequate $(d_H)_{(N)}$-$\tfrac{\delta}{2}$-chains as exposed in the previous reasoning, then one can easily deduce that so is the extended fuzzy dynamical system $(\Fc_{\infty}(X),\fuz{f})$.\newpage
	
	(iv) $\Rightarrow$ (v) $\Rightarrow$ (vi): This trivially follows from the the fact that $d_{S}(u,v) \leq d_{0}(u,v) \leq d_{\infty}(u,v)$ for every pair of fuzzy sets $u,v \in \Fc(X)$. Indeed, note that these inequalities imply that every $d_{\infty}$-$\delta$-chain is a $d_{0}$-$\delta$-chain, and that every $d_{0}$-$\delta$-chain is a $d_{S}$-$\delta$-chain (see Proposition~\ref{Pro:fuzzy.metrics}).
	
	(vi) $\Rightarrow$ (iii): We start observing that: given two (not necessarily distinct) fuzzy sets $u,v \in \Fc(X)$, any $\delta>0$, and any $d_{S}$-$\delta$-chain $(u^j)_{j=0}^n$ from $u$ to $v$, then the sequence of supports $(u^j_0)_{j=0}^n$ is a $d_H$-$\delta$-chain in $\Kc(X)$. Indeed, by the condition of $d_{S}$-$\delta$-chain and Propositions~\ref{Pro:fuz{f}}~and~\ref{Pro:fuzzy.metrics} we have that
	\[
	d_H(\com{f}(u^j_0),u^{j+1}_0) \leq d_{S}(\fuz{f}(u^j),u^{j+1}) < \delta \quad \text{ for all } 0 \leq j \leq n-1.
	\]
	Thus, if the system $(\Fc_{S}(X),\fuz{f})$ is chain recurrent, chain transitive or chain mixing respectively, then so is the system $(\Kc(X),\com{f})$. In fact, given two (not necessarily distinct) compact sets $K,L \in \Kc(X)$, by the previous observation we have that: when there exists a $d_{S}$-$\delta$-chain from $\chi_{K}$ to $\chi_{L}$ and of length $n \in \NN$, then there exists a $d_H$-$\delta$-chain from $K$ to $L$ of the same length $n \in \NN$.
\end{proof}

\begin{remark}
	With Theorem~\ref{The:chains} we have generalized, for the context of dynamical systems acting on non-compact metric spaces, the results \cite[Corollary~3.2]{AyalaDaRo2018_arXiv_dynamics}, \cite[Equivalences~(A2),~(D1),~(D2)~and~(E1)]{FernandezGoodPulRa2015_CSF_chain} and \cite[Propositions~4.1,~4.2,~4.3,~4.5,~4.4~and~4.6,~Theorems~4.7~and~4.8,~and~Corollaries~4.9~and~4.11]{KhanKu2013_FEJDS_recurrence}.
\end{remark}

\subsection{Chain recurrence and the endograph metric}

As we advanced at the Introduction and at the beginning of Section~\ref{Sec_4:chains}, our second main result in this part of the paper is proving that the chain-type properties previously introduced behave in an extremely radical manner for the endograph metric. In particular, our result reads as follows:

\begin{theorem}\label{The:chains.E}
	Let $f:X\longrightarrow X$ be a continuous map on a metric space $(X,d)$. Then, the following statements are equivalent:
	\begin{enumerate}[{\em(i)}]
		\item the map $f:(X,d)\longrightarrow(X,d)$ has dense range;
		
		\item the system $(\Fc_{E}(X),\fuz{f})$ is chain recurrent;
		
		\item the system $(\Fc_{E}(X),\fuz{f})$ is chain transitive;
		
		\item the system $(\Fc_{E}(X),\fuz{f})$ is chain mixing.
	\end{enumerate}
\end{theorem}

The rest of the section is devoted to prove Theorem~\ref{The:chains.E}, whose proof relies on Lemma~\ref{Lem:dense.range}.

\begin{lemma}\label{Lem:dense.range}
	Let $f:X\longrightarrow X$ be a continuous map on a metric space $(X,d)$. Then, the following statements are equivalent:
	\begin{enumerate}[{\em(i)}]
		\item the map $f:X\longrightarrow X$ has $d$-dense range, i.e.\ $f(X)$ is dense in $(X,d)$;
		
		\item the map $\com{f}:\Kc(X)\longrightarrow\Kc(X)$ has $d_H$-dense range, i.e.\ $\com{f}(\Kc(X))$ is dense in $(\Kc(X),d_H)$;
		
		\item the map $\fuz{f}:\Fc(X)\longrightarrow\Fc(X)$ has $d_{\infty}$-dense range, i.e.\ $\fuz{f}(\Fc(X))$ is dense in $(\Fc(X),d_{\infty})$;
		
		\item the map $\fuz{f}:\Fc(X)\longrightarrow\Fc(X)$ has $d_{0}$-dense range, i.e.\ $\fuz{f}(\Fc(X))$ is dense in $(\Fc(X),d_{0})$;
		
		\item the map $\fuz{f}:\Fc(X)\longrightarrow\Fc(X)$ has $d_{S}$-dense range, i.e.\ $\fuz{f}(\Fc(X))$ is dense in $(\Fc(X),d_{S})$;
		
		\item the map $\fuz{f}:\Fc(X)\longrightarrow\Fc(X)$ has $d_{E}$-dense range, i.e.\ $\fuz{f}(\Fc(X))$ is dense in $(\Fc(X),d_{E})$.
	\end{enumerate}
\end{lemma}
Equivalence (i) $\Leftrightarrow$ (ii) may be folklore, but we include its brief proof for the sake of completeness.
\begin{proof}[Proof of Lemma~\ref{Lem:dense.range}]
	(i) $\Leftrightarrow$ (ii): Assume that $f$ has $d$-dense range and let us prove that $\com{f}(\Kc(X))$ cuts every $d_H$-open set. In fact, since every $d_H$-open set contains a basic Vietoris-open set of the form
	\[
	\Vc(U_1,U_2,...,U_N) = \left\{ K \in \Kc(X) \ ; \ K \subset \bigcup_{j=1}^N U_j \text{ and } K\cap U_j\neq\varnothing \text{ for all } 1\leq j\leq N \right\}
	\]
	for some $N \in \NN$ and $d$-open sets $U_1,U_2,...,U_N \subset X$, picking a point $x_j \in f^{-1}(U_j)$ for each $1 \leq j \leq N$ and setting $L:=\{x_1,x_2,...,x_N\} \in \Kc(X)$ it follows that $\com{f}(L) \in \Vc(U_1,U_2,...,U_N)$.
	
	Conversely, if we assume that the map $\com{f}$ has $d_H$-dense range, then for each arbitrary non-empty $d$-open subset $U \subset X$ we can consider the basic Vietoris-open set $\Vc(U) = \{ K \in \Kc(X) \ ; \ K \subset U \}$ and a compact set $L' \in \Kc(X)$ fulfilling that $\com{f}(L') \in \Vc(U)$. Picking any point $x \in L'$ it follows that $f(x) \in \com{f}(L') \subset U$, which implies that the map $f$ has $d$-dense range by the arbitrariness of $U$.
	
	(ii) $\Rightarrow$ (iii): Assume that $\com{f}$ has $d_H$-dense range, fix any $u \in \Fc(X)$ and any $\eps>0$, and let us prove that $\fuz{f}(\Fc(X))$ cuts the $d_{\infty}$-open ball $\Bc_{\infty}(u,\eps)$. In fact, using Lemma~\ref{Lem:eps.pisos} we can find a finite sequence of numbers $0 = \alpha_0 < \alpha_1 < \alpha_2 < ... < \alpha_N = 1$ such that
	\[
	d_H(u_{\alpha}, u_{\alpha_{l+1}})<\tfrac{\eps}{2} \quad \text { for all } \alpha \in \ ]\alpha_l, \alpha_{l+1}] \text{ and } 0\leq l\leq N-1,
	\]
	but also fulfilling that $d_H(u_0,u_{\alpha_1})<\tfrac{\eps}{2}$. By assumption, there exist compact sets $K_l \in \Kc(X)$ fulfilling that $d_H(u_{\alpha_l},\com{f}(K_l))<\tfrac{\eps}{2}$ for each $1 \leq l \leq N$. Considering the fuzzy set $v := \max_{1\leq l\leq N} \left( \alpha_l \cdot \chi_{K_l} \right)$, we are going to check that $\fuz{f}(v) \in \Bc_{\infty}(u,\eps)$. Indeed, by statement (b) of Proposition~\ref{Pro:Hausdorff} we have that
	\[
	d_H\left(u_{\alpha_l},\com{f}(v_{\alpha_l})\right) = d_H\left( \bigcup_{l \leq j\leq N} u_{\alpha_j}, \bigcup_{l \leq j \leq N} \com{f}(K_j) \right) \leq \max_{1\leq j\leq N} d_H\left( u_{\alpha_j} ,\com{f}(K_j) \right) < \tfrac{\eps}{2}
	\]
	whenever $1\leq l\leq N$. Thus, given any $\alpha \in [0,1]$ we have that
	\begin{align*}
		d_H\left( u_{\alpha} , \left[\fuz{f}(v)\right]_{\alpha} \right) &= d_H\left( u_{\alpha} , \com{f}(v_{\alpha}) \right)
		= \left\{
		\begin{array}{lcc}
			d_H\left( u_{\alpha} , \com{f}(v_{\alpha_1}) \right), & \text{ if } \alpha \in [0,\alpha_1] \hspace{3.05cm} \\[5pt]
			d_H\left( u_{\alpha} , \com{f}(v_{\alpha_{l+1}}) \right), & \text{ if } \alpha \in \ ]\alpha_l,\alpha_{l+1}], 1\leq l\leq N-1
		\end{array}
		\right\} \\[7.5pt]
		&\leq \left\{
		\begin{array}{lcc}
			d_H\left( u_{\alpha} , u_{\alpha_1} \right) + d_H\left( u_{\alpha_1} , \com{f}(v_{\alpha_1}) \right), & \text{ if } \alpha \in [0,\alpha_1] \hspace{3.05cm} \\[5pt]
			d_H\left( u_{\alpha} , u_{\alpha_{l+1}} \right) + d_H\left( u_{\alpha_{l+1}} , \com{f}(v_{\alpha_{l+1}}) \right), & \text{ if } \alpha \in \ ]\alpha_l,\alpha_{l+1}], 1\leq l\leq N-1
		\end{array}
		\right\} \\[7.5pt]
		&< \tfrac{\eps}{2} + \max_{1\leq l\leq N} d_H\left(u_{\alpha_l},\com{f}(v_{\alpha_l})\right) < \tfrac{\eps}{2} +\tfrac{\eps}{2} = \eps,
	\end{align*}
	which finally implies that $d_{\infty}(u,\fuz{f}(v))<\eps$. The arbitrariness of $u \in \Fc(X)$ and $\eps>0$ shows that the map $\fuz{f}$ has $d_{\infty}$-dense range, finishing the proof.
	
	(iii) $\Rightarrow$ (iv) $\Rightarrow$ (v) $\Rightarrow$ (vi): This trivially follows from the inclusions $\tau_{E} \subset \tau_{S} \subset \tau_{0} \subset \tau_{\infty}$.
	
	(vi) $\Rightarrow$ (ii): Assume that $\fuz{f}$ has $d_{E}$-dense range, fix any $K \in \Kc(X)$ and any $\eps>0$, and let us prove that $\com{f}(\Kc(X))$ cuts the $d_H$-open ball $\Bc_H(K,\eps)$. To simplify the proof, without loss of generality we will assume that $0<\eps\leq\tfrac{1}{2}$. By assumption, there exists a fuzzy set $u \in \Fc(X)$ fulfilling that $\fuz{f}(u)$ belongs to the $d_{E}$-open ball $\Bc_{E}(\chi_{K},\eps)$. Letting
	\[
	\delta := d_{E}(\chi_{K},\fuz{f}(u)) < \eps \leq \tfrac{1}{2},
	\]
	and picking any $\alpha \in \ ]\delta,1-\delta]$, it follows from Lemma~\ref{Lem:key} that $d_H(K,\com{f}(u_{\alpha})) \leq \delta < \eps$, so that $u_{\alpha}$ is an element of $\Kc(X)$ fulfilling that $\com{f}(u_{\alpha}) \in \Bc_H(K,\eps)$. The arbitrariness of $K \in \Kc(X)$ and $\eps>0$ shows that $\com{f}$ has $d_H$-dense range, finishing the proof.
\end{proof}

\begin{proof}[Proof of Theorem~\ref{The:chains.E}]
	The implications (iv) $\Rightarrow$ (iii) $\Rightarrow$ (ii) hold by the definition of chain mixing, chain transitivity and chain recurrence. Moreover, since the underlying map of every chain recurrent dynamical system has dense range, statement (ii) implies that $\fuz{f}$ has $d_{E}$-dense range, which in its turn implies that $f$ has $d$-dense range by Lemma~\ref{Lem:dense.range}, i.e.\ (ii) $\Rightarrow$ (i). Thus, to complete the proof we only need to show that (i) $\Rightarrow$ (iv). From now on assume that $f$ has $d$-dense range, note that then $\fuz{f}$ has $d_{E}$-dense range by Lemma~\ref{Lem:dense.range}, and let us prove the following strengthened version of (iv):
	\begin{enumerate}[--]
		\item \textit{for each $\delta>0$ there exists a positive integer $n_{\delta} \in \NN$ such that, given any pair of sets $u,v \in \Fc(X)$ and any integer $n \geq 2n_{\delta}$, then there exists a $d_{E}$-$\delta$-chain from $u$ to $v$ of length $n$ for the map $\fuz{f}$.}
	\end{enumerate}
	To do so fix any $\delta>0$ and let $n_{\delta} := \lfloor\tfrac{1}{\delta}\rfloor+1$, where $\lfloor\tfrac{1}{\delta}\rfloor$ denotes the floor of $\tfrac{1}{\delta}$. In addition, let us fix the positive value $\eps := \tfrac{1}{n_{\delta}}$, which clearly fulfills that $0<\eps<\delta$. Now, given any pair of arbitrary but fixed normal fuzzy sets $u,v \in \Fc(X)$ and any integer $n \geq 2n_{\delta}$ we can:
	\begin{enumerate}[--]
		\item fix a point $x \in u_1$, which exists because $u \in \Fc(X)$ and hence $u_1 = u^{-1}(\{1\}) \neq \varnothing$;
		
		\item and fixing a fuzzy set $w \in \Fc(X)$ fulfilling that $\fuz{f}^n(w) \in \Bc_{E}(v,\delta-\eps)$, which exists since the map $\fuz{f}^n$ has $d_{E}$-dense range because $\fuz{f}$ does by assumption and Lemma~\ref{Lem:dense.range}.
	\end{enumerate}
	Then, consider the finite sequence of functions $(u^j:X\longrightarrow\II)_{j=0}^n$ formed by
	\[
	u^j := \max\{ \chi_{\{f^j(x)\}} , (1-j\eps) \cdot \fuz{f}^j(u) , j\eps \cdot \fuz{f}^j(w) \} \quad \text{ for each } 0\leq j\leq n_{\delta},
	\]
	\[
	u^j := \max\{ \fuz{f}^j(w) , (1-(j-n_{\delta})\eps) \cdot \chi_{\{f^j(x)\}} \} \quad \text{ for each } n_{\delta}+1\leq j\leq 2n_{\delta},
	\]
	and
	\[
	u^j := \fuz{f}^j(w) \quad \text{ for each } 2n_{\delta}+1\leq j\leq n \text{ whenever } n > 2n_{\delta}.
	\]
	We start by noticing that all these functions are normal fuzzy sets: actually, the maximum of finitely many upper-semicontinuous functions is again an upper-semicontinuous function; all of these functions clearly have compact support; and these functions attain the value $1$ by definition. Now, given any integer $0 \leq j \leq n-1$ we will check that $d_{E}(\fuz{f}(u^j),u^{j+1}) = \com{d}_H(\eend(\fuz{f}(u^j),\eend(u^{j+1})) \leq \eps < \delta$. In fact:
	\begin{enumerate}[--]
		\item \textbf{Case 1}: \textit{When $0\leq j\leq n_{\delta}-1$}. In this case the reader can easily check that
		\[
		\fuz{f}(u^j) = \max\{ \chi_{\{f^{j+1}(x)\}} , (1-j\eps) \cdot \fuz{f}^{j+1}(u) , j\eps \cdot \fuz{f}^{j+1}(w) \}.
		\]
		Thus, if for each $\beta \in [0,1]$ we write
		\begin{align*}
			U_{\beta} &:= \{ (t,\alpha) \in X\times\II \ ; \ \beta \cdot [\fuz{f}^{j+1}(u)](t) \geq \alpha \}, \\[5pt]
			W_{\beta} &:= \{ (t,\alpha) \in X\times\II \ ; \ \beta \cdot [\fuz{f}^{j+1}(w)](t) \geq \alpha \},
		\end{align*}
		we than have that
		\[
		\eend(\fuz{f}(u^j)) = \{f^{j+1}(x)\}\times\II \ \cup \ U_{1-j\eps} \ \cup \ W_{j\eps},
		\]
		while
		\[
		\eend(u^{j+1}) = \{f^{j+1}(x)\}\times\II \ \cup \ U_{1-(j+1)\eps} \ \cup \ W_{(j+1)\eps}.
		\]
		Since in the space of closed sets $\Cc(X\times\II)$ with the distance $\com{d}_H$ one can trivially check that
		\[
		U_{1-(j+1)\eps} \subset U_{1-j\eps} \subset U_{1-(j+1)\eps} + \eps \quad \text{ and that } \quad W_{j\eps} \subset W_{(j+1)\eps} \subset W_{j\eps} + \eps,
		\]
		it follows that $\eend(\fuz{f}(u^j)) \subset \eend(u^{j+1}) + \eps$ but also that $\eend(u^{j+1}) \subset \eend(\fuz{f}(u^j)) + \eps$. Hence, using statement (a) of Proposition~\ref{Pro:Hausdorff} we get that $d_{E}(\fuz{f}(u^j),u^{j+1}) = \com{d}_H(\eend(\fuz{f}(u^j),\eend(u^{j+1})) \leq \eps < \delta$.
		
		\item \textbf{Case 2}: \textit{When $n_{\delta} \leq j\leq 2n_{\delta}-1$}. In this case the reader can easily check that
		\[
		\fuz{f}(u^j) = \max\{ \fuz{f}^{j+1}(w) , (1-(j-n_{\delta})\eps) \cdot \chi_{\{f^{j+1}(x)\}} \}.
		\]
		Thus, we have that
		\[
		\eend(\fuz{f}(u^j)) = \eend(\fuz{f}^{j+1}(w)) \ \cup \ \{f^{j+1}(x)\}\times[0,1-(j-n_{\delta})\eps],
		\]
		while
		\[
		\eend(u^{j+1}) = \eend(\fuz{f}^{j+1}(w)) \ \cup \ \{f^{j+1}(x)\}\times[0,1-((j+1)-n_{\delta})\eps].
		\]
		Note that $\eend(u^{j+1}) \subset \eend(\fuz{f}(u^j))$. Moreover, since in the space of closed sets $\Cc(X\times\II)$ with the distance $\com{d}_H$ one can trivially check that
		\[
		\{f^{j+1}(x)\}\times[0,1-(j-n_{\delta})\eps] \subset \{f^{j+1}(x)\}\times[0,1-((j+1)-n_{\delta})\eps] + \eps,
		\]
		it follows that $\eend(\fuz{f}(u^j)) \subset \eend(u^{j+1}) + \eps$. Hence, using statement (a) of Proposition~\ref{Pro:Hausdorff} we obtain that $d_{E}(\fuz{f}(u^j),u^{j+1}) = \com{d}_H(\eend(\fuz{f}(u^j),\eend(u^{j+1})) \leq \eps < \delta$.
		
		\item \textbf{Case 3}: \textit{When $n>2n_{\delta}$ and $2n_{\delta} \leq j < n-1$}. This case is completely trivial since we have the equalities $\fuz{f}(u^j) = \fuz{f}^{j+1}(w) = u^{j+1}$, so that $d_{E}(\fuz{f}(u^j),u^{j+1}) = 0 \leq \eps < \delta$.
	\end{enumerate}
	Since $u^0 = u$ and $u^n = \fuz{f}^n(w)$, we have thus constructed a $d_{E}$-$\delta$-chain $(u^j)_{j=0}^n$ from $u$ to $\fuz{f}^n(w)$ of length $n$ fulfilling that $d_{E}(\fuz{f}(u^{n-1}),u^n) \leq \eps$. Hence, since $d_{E}(\fuz{f}^n(w),v) < \delta-\eps$ we have that
	\[
	d_{E}(\fuz{f}(u^{n-1}),v) \leq d_{E}(\fuz{f}(u^{n-1}),u^n) + d_{E}(u^n,v) \leq \eps + d_{E}(\fuz{f}^n(w),v) < \delta,
	\]
	and replacing $u^n$ with $v$ in $(u^j)_{j=0}^n$ we obtain a $d_{E}$-$\delta$-chain from $u$ to $v$ of length $n$ for the map $\fuz{f}$.
\end{proof}

\begin{remark}
	Let us continue the comparison of ``\textit{chain recurrence}'' with ``\textit{topological recurrence}'', of the notion of~``\textit{chain transitivity}'' with that of ``\textit{topological transitivity}'' and of the property of being ``\textit{chain mixing}'' with that of being ``\textit{topologically mixing}'' mentioned at Subsection~\ref{SubSec_4.1:chains}. Actually, the reader may find interesting to compare Theorems~\ref{The:chains}~and~\ref{The:chains.E}, that we have proved for the respective chain-type properties, with \cite[Theorems~3.1 and 3.3]{Lopez2026_IJFS_topological-I}, where it was showed that the system $(\Kc(X),\com{f})$ is topologically recurrent, topologically transitive or topologically mixing respectively, if and only if so do $(\Fc_{\infty}(X),\fuz{f})$, $(\Fc_{0}(X),\fuz{f})$, $(\Fc_{S}(X),\fuz{f})$ and $(\Fc_{E}(X),\fuz{f})$. These results provide, once again, an example of how the endograph metric can behave in an extremely radical way: while it aligns well with the topological notions, its behaviour changes dramatically for chains and $\delta$-perturbations.
\end{remark}

\section{The shadowing property}\label{Sec_5:shadowing}

In this section we focus on the shadowing and finite shadowing properties. These dynamical notions were considered in \cite{BartollMaPeRo2022_AXI_orbit} for the fuzzy dynamical systems $(\Fc_{\infty}(X),\fuz{f})$ and $(\Fc_{0}(X),\fuz{f})$. However, their main result on shadowing \cite[Theorem~5]{BartollMaPeRo2022_AXI_orbit} is not completely valid. We therefore have divided this section into two parts: first we provide two counterexamples together with a corrected version of the statement of such a result (see Examples~\ref{Exa_1:discrete}~and~\ref{Exa_2:connected}~and~Theorem~\ref{The:shadowing}); and then we prove that $(\Fc_{E}(X),\fuz{f})$ cannot have the finite shadowing property when the map~$f$ has dense range and~$(X,f)$ is not topologically mixing, but that $(\Fc_{E}(X),\fuz{f})$ has the full shadowing property when~$(X,f)$ is contractive and $f^k(X)$ is bounded for some $k \in \NN$ (see Theorems~\ref{The:shadowing.E}~and~\ref{The:contractions}). The reader may observe along this section that, as it happens in Section~\ref{Sec_3:contraexpansive} with the contractive and expansive notions, the metrics $d_{0}$ and $d_{S}$ together with their respective systems~$(\Fc_{0}(X),\fuz{f})$ and $(\Fc_{S}(X),\fuz{f})$ exhibit, for the shadowing property, a behaviour similar to that of $d_{E}$ and its corresponding dynamical system $(\Fc_{E}(X),\fuz{f})$.

\subsection{Definitions, counterexamples and equivalences}\label{SubSec_5.1:equivalences}

Shadowing forms part of the so-called ``orbit tracing properties'' in Topological Dynamics. In fact, this dynamical notion was motivated by questions such as if an approximate trajectory of a dynamical system can be fitted or not by a real trajectory of the system. Thus, let us start by recalling what an~``approximate trajectory'' is. Given a continuous map $f:X\longrightarrow X$ acting on a metric space $(X,d)$, two positive numbers $\eps,\delta>0$, and an infinite sequence $(x_j)_{j=0}^{\infty}$ in $X$, we say that:
\begin{enumerate}[--]
	\item the sequence $(x_j)_{j=0}^{\infty}$ is a {\em $d$-$\delta$-pseudo-trajectory} for the map $f$ if
	\[
	d(f(x_j),x_{j+1})<\delta \quad \text{ for all } 0\leq j< \infty;
	\]
	
	\item the sequence $(x_j)_{j=0}^{\infty}$ is {\em $d$-$\eps$-shadowed} by the $f$-orbit of some point $x \in X$ if
	\[
	d(f^j(x),x_j)<\eps \quad \text{ for all } 0\leq j< \infty.
	\]
\end{enumerate}
A \textit{$\delta$-pseudo-trajectory} can be seen as an infinite \textit{$\delta$-chain} from those considered in Section~\ref{Sec_4:chains}, and we will say that a finite sequence $(x_j)_{j=0}^n$ of points in $X$ is {\em $d$-$\eps$-shadowed} by the $f$-orbit of some point $x \in X$  whenever $d(f^j(x),x_j)<\eps$ for all $0\leq j\leq n$. Following \cite{BartollMaPeRo2022_AXI_orbit,FernandezGood2016_FM_shadowing,GoodMitTho2020_JMAA_preservation} we will say that:
\begin{enumerate}[--]
	\item a dynamical system $(X,f)$ on a metric space $(X,d)$ has the (resp.\ {\em finite}) {\em shadowing property} if for any $\eps>0$ there exists some $\delta_{\eps}>0$ such that each $d$-$\delta_{\eps}$-pseudo-trajectory (resp.\ $d$-$\delta_{\eps}$-chain) for the map $f$ is $d$-$\eps$-shadowed by the $f$-orbit of some point in $X$.
\end{enumerate}
The reader should note that every dynamical system with the ``full'' shadowing property has the finite shadowing property. The converse does not hold in general (see \cite[Section~2]{BernardesPe2025_AdvM_on}), but it does for compact metric spaces (see \cite[Remark~1]{BarwellGoodORai2013_DCDS_characterizations}). Moreover, it was proved in~\cite[Theorem~3.4]{FernandezGood2016_FM_shadowing} that \textit{when $f:X\longrightarrow X$ is a continuous map on a \textbf{compact} metric space $(X,d)$, then the system $(X,f)$ has the shadowing property if and only if so does $(\Kc(X),\com{f})$}. If the \textbf{compactness} assumption is dropped, then one still gets the equivalence for the {\em finite shadowing property}. In fact, the proof of \cite[Theorem~3.4]{FernandezGood2016_FM_shadowing} follows from three results: first we have \cite[Lemma~3.1]{FernandezGood2016_FM_shadowing}, where compactness can be dropped by properly using continuity rather than uniform continuity; and then we have \cite[Theorems~3.2~and~3.3]{FernandezGood2016_FM_shadowing}, which remain valid exactly with the same proof when ``shadowing'' is replaced with ``finite shadowing''.

For the fuzzy extension $(\Fc(X),\fuz{f})$, it was claimed in \cite[Theorem~5]{BartollMaPeRo2022_AXI_orbit} that $(X,f)$ and $(\Kc(X),\com{f})$ have the finite shadowing property if and only if the systems $(\Fc_{\infty}(X),\fuz{f})$ and $(\Fc_{0}(X),\fuz{f})$ do as well. As we shall reformulate in Theorem~\ref{The:shadowing} below, this assertion is only partially correct: while the equivalence holds for $(\Fc_{\infty}(X),\fuz{f})$, it fails for $(\Fc_{0}(X),\fuz{f})$, as the following counterexamples exhibits.

\begin{example}[\textbf{A discrete counterexample}]\label{Exa_1:discrete}
	\textit{There exists a system $(X,f)$ having the shadowing property for which $(\Fc_{0}(X),\fuz{f})$, $(\Fc_{S}(X),\fuz{f})$ and $(\Fc_{E}(X),\fuz{f})$ do not have the finite shadowing property.}
	\begin{proof}
		Let $X=\{a,b\}$ with the discrete metric described in Example~\ref{Exa:expanding}, and let $f:X\longrightarrow X$ be the identity map, that is, $f(a)=a$ and $f(b)=b$. It is clear that $(X,f)$ has the shadowing property by taking $\delta_{\eps}:=\eps$ for any $\eps>0$. Let us now prove that there is an $\eps_0>0$ such that for any positive value $\delta>0$ there exists a $d_{0}$-$\delta$-chain that is not $d_{E}$-$\eps_0$-shadowed. This fact will be enough to conclude the counterexample: actually, since the inequalities $d_{E}(u,v) \leq d_{S}(u,v) \leq d_{0}(u,v)$ hold for every pair $u,v \in \Fc(X)$ as stated in Proposition~\ref{Pro:fuzzy.metrics}, we then have that every $d_{0}$-$\delta$-chain is also a $d_{S}$-$\delta$-chain and hence a $d_{E}$-$\delta$-chain, but also if a finite sequence cannot be $d_{E}$-$\eps_0$-shadowed by an $\fuz{f}$-orbit in $\Fc(X)$ then such a finite sequence cannot be $d_{S}$-$\eps_0$-shadowed, nor $d_{0}$-$\eps_0$-shadowed, by an $\fuz{f}$-orbit in $\Fc(X)$.
		
		We start by fixing any positive value $0<\eps_0<\tfrac{1}{4}$ and any $\delta>0$. Now, let $k \in 2\NN$ be an even integer fulfilling that $\tfrac{1}{k}<\min\{\delta,\tfrac{1}{4}-\eps_0\}$, let $n=\tfrac{k}{2}-1$ and consider the sequence $(u^j)_{j=0}^n$ of length $n$ with
		\[
		u^j := \chi_{\{a\}} + (\tfrac{1}{2} + \tfrac{j}{k}) \chi_{\{b\}} \quad \text{ for each } 0\leq j \leq n.
		\]
		Note that we have the equalities $u^0 = \chi_{\{a\}} + \tfrac{1}{2} \chi_{\{b\}}$ and $u^{n} = \chi_{\{a\}} + (1-\tfrac{1}{k}) \chi_{\{b\}}$. Let us now check:
		\begin{enumerate}[--]
			\item \textbf{Fact 1}: \textit{The sequence $(u^j)_{j=0}^n$ is a $d_{0}$-$\delta$-chain from $u^0$ to $u^n$}. Since $\fuz{f}$ we is the identity map in~$\Fc(X)$, we only have to check that $d_{0}(u^j,u^{j+1})<\delta$ for all $0\leq j \leq n-1$. We argue as in Lemma~\ref{Lem:contraexpansive}: fix any $0 \leq j \leq n-1$ and consider the 2-piecewise-linear map
			\begin{equation}\label{eq:xi_j}
				\xi_j(\alpha) :=
				\begin{cases}
					\tfrac{k+2j}{k+2(j+1)}\alpha & \text{ if } 0\leq \alpha\leq \tfrac{1}{2}+\tfrac{j+1}{k},\\[5pt]
					\tfrac{k-2j}{k-2(j+1)}\alpha - \tfrac{2}{k-2(j+1)} & \text{ if } \tfrac{1}{2}+\tfrac{j+1}{k}< \alpha\leq 1,
				\end{cases}
			\end{equation}
			which fulfills that $\xi_j(0)=0$, that $\xi_j(\tfrac{1}{2}+\tfrac{j+1}{k})=\tfrac{1}{2}+\tfrac{j}{k}$ and that $\xi_j(1)=1$. It is not hard to see that
			\begin{equation}\label{eq:xi_j=1/k}
				\sup_{\alpha\in\II} |\xi_j(\alpha)-\alpha| = \left|\xi_j\left(\tfrac{1}{2}+\tfrac{j+1}{k}\right) - \left(\tfrac{1}{2}+\tfrac{j+1}{k}\right)\right| = \left|\left(\tfrac{1}{2}+\tfrac{j}{k}\right) - \left(\tfrac{1}{2}+\tfrac{j+1}{k}\right)\right| = \tfrac{1}{k},
			\end{equation}
			and that $\xi_k^{-1}(\tfrac{1}{2}+\tfrac{j}{k}) = \tfrac{1}{2}+\tfrac{j+1}{k}$. It follows that
			\[
			\left[ \xi_j \circ u^{j+1} \right]_{\alpha} = \left[ u^{j+1} \right]_{\xi_j^{-1}(\alpha)} =
			\begin{cases}
				\{a,b\} & \text{ if } 0\leq \alpha\leq \tfrac{1}{2}+\tfrac{j}{k},\\[5pt]
				\{a\} & \text{ if } \tfrac{1}{2}+\tfrac{j}{k}<\alpha\leq 1.
			\end{cases}
			\]
			This implies that $u^j = \xi_j \circ u^{j+1}$ and hence $d_{\infty}(u^j,\xi_j \circ u^{j+1})=0$. We finally deduce that
			\[
			d_{0}(u^j,u^{j+1}) \leq \max\left\{ \sup_{\alpha\in\II} |\xi_j(\alpha)-\alpha| \ , \ d_{\infty}(u^j,\xi_j \circ u^{j+1}) \right\} \overset{\eqref{eq:xi_j=1/k}}{=} \tfrac{1}{k} < \delta.
			\]
			
			\item \textbf{Fact 2}: \textit{The $d_{0}$-$\delta$-chain $(u^j)_{j=0}^n$ is not $d_{E}$-$\eps_0$-shadowed by any $\fuz{f}$-orbit in $\Fc(X)$}. By contradiction, assume that such a sequence is $d_{E}$-$\eps_0$-shadowed by the orbit of some $u \in \Fc(X)$, which means that
			\[
			d_{E}(u,u^j) = d_{E}(\fuz{f}^j(u),u^j) < \eps_0 \quad \text{ for all } 0 \leq j \leq n,
			\]
			because $\fuz{f}$ is the identity map. Hence, from the fact that $\com{d}_H(\eend(u),\eend(u^0)) = d_{E}(u,u^0) < \eps_0 < \tfrac{1}{4}$ we have that $\eend(u) \subset \eend(u^0) + \eps_0 \subset \eend(u^0) + \tfrac{1}{4}$, where
			\[
			\eend(u^0) + \tfrac{1}{4} = \{a\}\times[0,1] \ \cup \ \{b\}\times[0,\tfrac{3}{4}].
			\]
			This necessarily implies that $u(a)=1$ and that $u(b) \in [0,\tfrac{3}{4}]$. Hence, although we also have that $\com{d}_H(\eend(u),\eend(u^n)) = d_{E}(u,u^n) < \eps_0$, we can now check that $\eend(u^n) \not\subset \eend(u) + \eps_0$, reaching a contradiction since then $\com{d}_H(\eend(u),\eend(u^n))>\eps_0$ by statement~(a) of Proposition~\ref{Pro:Hausdorff}. In fact, note that the point $(b,1-\tfrac{1}{k}) \in \eend(u^n)$ fulfills that
			\[
			\inf_{(x,\alpha) \in \eend(u)} \com{d}( (b,1-\tfrac{1}{k}) , (x,\alpha) ) \geq \inf_{(x,\alpha) \in \eend(u^0) + \tfrac{1}{4}} \max\{ d(b,x) , |(1-\tfrac{1}{k})-\alpha| \},
			\]
			and for each $(x,\alpha) \in \eend(u^0) + \tfrac{1}{4}$ we have that
			\begin{align*}
				\max\{ d(b,x) , |(1-\tfrac{1}{k})-\alpha| \} &= \left\{
				\begin{array}{ll}
					\max\{ d(b,a) , |(1-\tfrac{1}{k})-\alpha| \} & \text{ if } 0\leq\alpha\leq 1 \text{ and } x=a, \\[5pt]
					\max\{ d(b,b) , |(1-\tfrac{1}{k})-\alpha| \} & \text{ if } 0\leq\alpha\leq \tfrac{3}{4} \text{ and } x=b,
				\end{array}
				\right\} \\[7.5pt]
				&\geq \left\{
				\begin{array}{ll}
					1 & \text{ if } 0\leq\alpha\leq 1 \text{ and } x=a, \\[5pt]
					|(1-\tfrac{1}{k})-\tfrac{3}{4}| & \text{ if } 0\leq\alpha\leq \tfrac{3}{4} \text{ and } x=b,
				\end{array}
				\right\} \geq \tfrac{1}{4}-\tfrac{1}{k} > \eps_0.\qedhere
			\end{align*}
		\end{enumerate}
	\end{proof}
\end{example}

The idea behind Example~\ref{Exa_1:discrete}, and the reason why we have taken the identity map, was to keep the arguments as simple as possible so that the reader can easily see why the shadowing property fails for the metrics $d_{0}$, $d_{S}$ and $d_{E}$. In fact, the construction of the system in Example~\ref{Exa_1:discrete} admits several variations: one may take $f:\{a,b\}\longrightarrow\{a,b\}$ with $f(a)=b$ and $f(b)=a$, but one could also consider a discrete space with $N \geq 3$ points, $X=\{a_1,a_2,...,a_N\}$, together with the map $f(a_j)=a_{j+1 \ (\text{mod } N)}$, among other plausible choices. In all these cases the same arguments apply. Since the reader might suspect that this phenomenon relies on the discreteness of the underlying metric space, we also provide a counterexample on a connected space, where completely analogous arguments hold.

\begin{example}[\textbf{A connected counterexample}]\label{Exa_2:connected}
	\textit{There exists a contractive dynamical system $(X,f)$ for which $(\Fc_{0}(X),\fuz{f})$, $(\Fc_{S}(X),\fuz{f})$ and $(\Fc_{E}(X),\fuz{f})$ do not have the finite shadowing property.}
	\begin{proof}
		Consider the set of real numbers $X=\RR$ with the usual metric $d(x,y)=|x-y|$ for each pair of points $x,y \in \RR$. Let $f:X\longrightarrow X$ be the contractive linear map with $f(x)=\tfrac{1}{2}x$ for each $x \in \RR$. It is a classical fact that the dynamical system $(X,f)$ has the shadowing property. One way to see this is by invoking the Linear Shadowing Lemma, which states that every hyperbolic invertible operator on a Banach space has the shadowing property (see \cite{Ombach1994_UIAM_the-shadowing}). Another approach is noticing that every contractive dynamical system has the shadowing property (see Proposition~\ref{Pro:contractions}). As in Example~\ref{Exa_1:discrete} we are going to prove that there is an $\eps_0>0$ such that for any $\delta>0$ there exists a $d_{0}$-$\delta$-chain that is not $d_{E}$-$\eps_0$-shadowed, which will complete the counterexample. To do so, we again fix any positive value $0<\eps_0<\tfrac{1}{4}$ and any $\delta>0$. Now, let $k \in 2\NN$ be an even integer fulfilling that $\tfrac{1}{k}<\min\{\delta,\tfrac{1}{4}-\eps_0\}$, let $n=\tfrac{k}{2}-1$ and consider the sequence $(u^j)_{j=0}^n$ of length $n$ with
		\[
		u^j := \chi_{\{0\}} + (\tfrac{1}{2} + \tfrac{j}{k}) \chi_{\{2^{n-j}\}} \quad \text{ for each } 0\leq j \leq n.
		\]
		It can be checked that we have the equalities $u^0 = \chi_{\{0\}} + \tfrac{1}{2} \chi_{\{2^n\}}$, $\fuz{f}(u^j) = \chi_{\{0\}} + (\tfrac{1}{2} + \tfrac{j}{k}) \chi_{\{2^{n-(j+1)}\}}$ for each $0\leq j \leq n-1$, and $u^{n} = \chi_{\{0\}} + (1-\tfrac{1}{k}) \chi_{\{1\}}$. Let us now check:
		\begin{enumerate}[--]
			\item \textbf{Fact 1}: \textit{The sequence $(u^j)_{j=0}^n$ is a $d_{0}$-$\delta$-chain from $u^0$ to $u^n$}. Actually, arguing as in Example~\ref{Exa_1:discrete}, fix any $0 \leq j \leq n-1$ and consider the 2-piecewise-linear map $\xi_j$ as defined in \eqref{eq:xi_j}. It follows that
			\[
			\left[ \xi_j \circ u^{j+1} \right]_{\alpha} = \left[ u^{j+1} \right]_{\xi_j^{-1}(\alpha)} =
			\begin{cases}
				\{0,2^{n-(j+1)}\} & \text{ if } 0\leq \alpha\leq \tfrac{1}{2}+\tfrac{j}{k},\\[5pt]
				\{0\} & \text{ if } \tfrac{1}{2}+\tfrac{j}{k}<\alpha\leq 1.
			\end{cases}
			\]
			Thus, $\fuz{f}(u^j) = \xi_j \circ u^{j+1}$ so that $d_{0}(\fuz{f}(u^j),u^{j+1}) \leq \tfrac{1}{k} < \delta$ by \eqref{eq:xi_j=1/k}, exactly as in Example~\ref{Exa_1:discrete}.
			
			\item \textbf{Fact 2}: \textit{The $d_{0}$-$\delta$-chain $(u^j)_{j=0}^n$ is not $d_{E}$-$\eps_0$-shadowed by any $\fuz{f}$-orbit in $\Fc(X)$}. By contradiction, assume that there exists some $u \in \Fc(X)$ fulfilling that $d_{E}(\fuz{f}^j(u),u^j)<\eps_0$ for all $0 \leq j \leq n$. Hence, since $\com{d}_H(\eend(u),\eend(u^0)) = d_{E}(u,u^0) < \eps_0 < \tfrac{1}{4}$ we have that
			\[
			\eend(u) \subset \eend(u^0) + \tfrac{1}{4} = [-\tfrac{1}{4},\tfrac{1}{4}]\times[0,1] \ \cup \ [2^n-\tfrac{1}{4},2^n+\tfrac{1}{4}]\times[0,\tfrac{3}{4}] \ \cup \ \RR\times[0,\tfrac{1}{4}].
			\]
			Applying $n$ times the map $\fuz{f}$ it is not hard to check that then
			\begin{equation}\label{eq:end(f^n(u)).included}
				\eend(\fuz{f}^n(u)) \subset [-\tfrac{1}{4\cdot2^n},\tfrac{1}{4\cdot2^n}]\times[0,1] \ \cup \ [1-\tfrac{1}{4\cdot2^n},1+\tfrac{1}{4\cdot2^n}]\times[0,\tfrac{3}{4}] \ \cup \ \RR\times[0,\tfrac{1}{4}].
			\end{equation}
			Hence, although we also have that $\com{d}_H(\eend(\fuz{f}^n(u)),\eend(u^n)) = d_{E}(\fuz{f}^n(u),u^n) < \eps_0$, we can now check that $\eend(u^n) \not\subset \eend(\fuz{f}^n(u)) + \eps_0$, reaching a contradiction with statement~(a) of Proposition~\ref{Pro:Hausdorff}. In fact, using that $\eps_0<\tfrac{1}{4}-\tfrac{1}{k}<\tfrac{1}{4}$ by assumption and hence that $\tfrac{1}{4\cdot2^n}+\eps_0<\tfrac{1}{2}$, the point $(1,1-\tfrac{1}{k})$ belongs to the set $\eend(u^n)$ while by \eqref{eq:end(f^n(u)).included} we have that $(1,1-\tfrac{1}{k}) \notin \eend(\fuz{f}^n(u))$ because
			\[
			\eend(\fuz{f}^n(u)) + \eps_0 \subset [-\tfrac{1}{2},\tfrac{1}{2}]\times[0,1] \ \cup \ [\tfrac{1}{2},\tfrac{3}{2}]\times[0,1-\tfrac{1}{k}[ \ \cup \ \RR\times[0,\tfrac{1}{2}].\qedhere
			\]
		\end{enumerate}
	\end{proof}
\end{example}

In Example~\ref{Exa_2:connected} we could equally have considered the expansive map $f(x)=2x$ on the same metric space~$(\RR,d)$, which also provides a counterexample where similar arguments apply (this system also has the shadowing property by the Linear Shadowing Lemma~\cite{Ombach1994_UIAM_the-shadowing}). Nevertheless, we chose to work with the contractive map $f(x)=\tfrac{1}{2}x$, since contractions will play a central role in Subsection~\ref{SubSec_5.2:contractions} below. Having established the previous counterexamples we are now in a position to adapt the ideas originally employed in \cite[Theorem~5]{BartollMaPeRo2022_AXI_orbit}, and we can formulate the next corrected version of it:

\begin{theorem}\label{The:shadowing}
	Let $f:X\longrightarrow X$ be a continuous map on a metric space $(X,d)$. Hence:
	\begin{enumerate}[{\em(a)}]
		\item The following statements are equivalent:
		\begin{enumerate}[{\em(i)}]
			\item $(X,f)$ has the finite shadowing property;
			
			\item $(\Kc(X),\com{f})$ has the finite shadowing property;
			
			\item $(\Fc_{\infty}(X),\fuz{f})$ has the finite shadowing property.
		\end{enumerate}
		Moreover, if $(X,d)$ is a compact metric space, then statements {\em(i)} and {\em(ii)} can be replaced by:
		\begin{enumerate}[{\em(i)}]
			\item[{\em(i')}] $(X,f)$ has the shadowing property;
			
			\item[{\em(ii')}] $(\Kc(X),\com{f})$ has the shadowing property.
		\end{enumerate}
		
		\item If the system $(\Kc(X),\com{f})$ has the shadowing property, the so does the system $(X,f)$.
		
		\item If any of the systems $(\Fc_{\infty}(X),\fuz{f})$, $(\Fc_{0}(X),\fuz{f})$, $(\Fc_{S}(X),\fuz{f})$ or $(\Fc_{E}(X),\fuz{f})$ has the (resp.\ finite) shadowing property, then so does the system $(\Kc(X),\com{f})$.
	\end{enumerate}
\end{theorem}
\begin{proof}
	(a): The equivalence (i) $\Leftrightarrow$ (ii) follows from the arguments exhibited in \cite[Theorem~3.4]{FernandezGood2016_FM_shadowing}, even if $(X,d)$ is not a compact space (see the comments at the beginning of Subsection~\ref{SubSec_5.1:equivalences}). For the particular case in which $(X,d)$ is compact, then $(\Kc(X),d_H)$ is compact by \cite[Theorem~3.5]{IllanesNad1999_book_hyperspaces} and the equivalences of statements (i)-(ii) with (i')-(ii') follow from \cite[Remark~1]{BarwellGoodORai2013_DCDS_characterizations}. The implication (ii) $\Rightarrow$ (iii) was proved in \cite[Theorem~5]{BartollMaPeRo2022_AXI_orbit} for the case in which $(X,d)$ is a compact metric space, but their proof does not use compactness. We prove the implication (iii) $\Rightarrow$ (ii) in part (c) of this theorem.
	
	(b): This was proved in \cite[Theorem~3.2]{FernandezGood2016_FM_shadowing} for the case in which $(X,d)$ is compact, but without using compactness. Thus, the proof given there is valid to conclude our statement.
	
	(c): Fix $\rho \in \{d_{\infty},d_{0},d_{S},d_{E}\}$ and assume that $(\Fc(X),\fuz{f})$ has the (resp.\ finite) shadowing property when $\Fc(X)$ is endowed with $\rho$. Thus, given $\eps>0$, which can be assumed to fulfill that $0<\eps\leq\tfrac{1}{2}$ without loss of generality, let $\delta_{\eps}>0$ such that every finite $\rho$-$\delta_{\eps}$-chain for $\fuz{f}$ can be $\rho$-$\tfrac{\eps}{2}$-shadowed by the $\fuz{f}$-orbit of some fuzzy set in $\Fc(X)$. Consider now a $d_H$-$\delta_{\eps}$-pseudo-trajectory (resp.\ $d_H$-$\delta_{\eps}$-chain) of compact sets $(K_j)_{j=0}^{n}$ with $n=\infty$ (resp.\ $n \in \NN$). Letting $u^j:=\chi_{K_j}$ for each possible $j$ we deduce that $(u^j)_{j=0}^n$ is a $\rho$-$\delta_{\eps}$-pseudo-trajectory (resp.\ $\rho$-$\delta_{\eps}$-chain). In fact, by Proposition~\ref{Pro:fuzzy.metrics} for $d_{\infty}$,
	\[
	\rho(\fuz{f}(u^j),u^{j+1}) \leq d_{\infty}(\fuz{f}(u^j),u^{j+1}) = d_H(\com{f}(K_j),K_{j+1})<\delta_{\eps} \quad \text{ for all } 0\leq j< n.
	\]
	By assumption there exists $u \in \Fc(X)$ such that $\rho(\fuz{f}^j(u),u^j) < \tfrac{\eps}{2} < \eps \leq \tfrac{1}{2}$ for all possible $j$. Using again Proposition~\ref{Pro:fuzzy.metrics}, but this time for $d_{E}$, we have that
	\[
	\eta := \sup_{j} d_{E}(\fuz{f}^j(u),u^j) \leq \sup_{j} \rho(\fuz{f}^j(u),u^j) \leq \tfrac{\eps}{2} < \eps \leq \tfrac{1}{2}.
	\]
	Picking now any $\alpha \in \ ]\eta,1-\eta]$, an application of Lemma~\ref{Lem:key} shows that
	\[
	d_H(\com{f}^j(u_{\alpha}),K_j) = d_H( [\fuz{f}^j(u)]_{\alpha} , K_j ) \leq \eta < \eps \quad \text{ for every possible } j.
	\]
	We deduce that $(K_j)_{j=0}^n$ is $d_H$-$\eps$-shadowed by the $\com{f}$-orbit of $u_{\alpha} \in \Kc(X)$, which completes the proof.
\end{proof}

\subsection{Shadowing for the Skorokhod, sendograph and endograph metrics}\label{SubSec_5.2:contractions}

In view of Theorem~\ref{The:contraexpansive} and given that the arguments employed in Lemma~\ref{Lem:contraexpansive} closely parallel those used in~Examples~\ref{Exa_1:discrete} and \ref{Exa_2:connected}, it is natural to ask whether $(\Fc_{0}(X),\fuz{f})$, $(\Fc_{S}(X),\fuz{f})$ or $(\Fc_{E}(X),\fuz{f})$ can exhibit the shadowing property whenever the set~$X$ contains more than one point and the map~$f$ is non-constant. We begin with a negative result for the endograph metric, derived from the theory developed in Section~\ref{Sec_4:chains}. Recall that a system $(X,f)$ is called {\em topologically mixing} if for any pair of non-empty open subsets $U,V \subset X$ there exists some $n_0 \in \NN$ such that $f^n(U)\cup V \neq \varnothing$ for all $n \geq n_0$.

\begin{theorem}\label{The:shadowing.E}
	Let $f:X\longrightarrow X$ be a continuous map on a metric space $(X,d)$ and assume that the map~$f$ has dense range in $(X,d)$. If the dynamical system $(X,f)$ is not topologically mixing, then the extended dynamical system $(\Fc_{E}(X),\fuz{f})$ cannot have the finite shadowing property.
\end{theorem}
\begin{proof}
	Assume that $(\Fc_{E}(X),\fuz{f})$ has the finite shadowing property and let us prove that then $(X,f)$ is topologically mixing. By \cite[Theorem~3.1]{Lopez2026_IJFS_topological-I} this is equivalent to prove that $(\Fc_{E}(X),\fuz{f})$ is topologically mixing. Moreover, since every $d_{E}$-open set contains an open $d_{E}$-ball, it is enough to prove that for any pair $u,v \in \Fc(X)$ and any $\eps>0$ there exists $n_0 \in \NN$ fulfilling that
	\[
	\fuz{f}^n(\Bc_{E}(u,\eps)) \cap \Bc_{E}(v,\eps) \neq \varnothing \quad \text{ for all } n\geq n_0.
	\]
	Actually, let $\delta_{\eps}>0$ so that every $d_{E}$-$\delta_{\eps}$-chain is $d_{E}$-$\eps$-shadowed by some $\fuz{f}$-orbit in $\Fc(X)$. Since the map $f:(X,d)\longrightarrow(X,d)$ has dense range, Theorem~\ref{The:chains.E} shows that $(\Fc_{E}(X),\fuz{f})$ is chain mixing. Thus, there is a positive integer $n_0 \in \NN$ such that for every $n \geq n_0$ there exists a $d_{E}$-$\delta$-chain $(u^{j,n})_{j=0}^n$ from $u$ to $v$ of length $n$. By assumption, for each $n \geq n_0$ there exists a fuzzy set $w^n \in \Fc(X)$ fulfilling that
	\[
	d_{E}(\fuz{f}^j(w^n),u^{j,n})<\eps \quad \text{ for all } 0\leq j\leq n.
	\]
	Thus, $d_{E}(w^n,u)=d_{E}(\fuz{f}^0(w^n),u^{0,n})<\eps$, $d_{E}(\fuz{f}^n(w^n),v)=d_{E}(\fuz{f}^n(w^n),u^{n,n})<\eps$ and hence
	\[
	\fuz{f}^n(w^n) \in \fuz{f}^n(\Bc_{E}(u,\eps)) \cap \Bc_{E}(v,\eps) \quad \text{ for each } n\geq n_0.\qedhere
	\]
\end{proof}

We can use Theorem~\ref{The:shadowing.E} to reprove that $(\Fc_{E}(X),\fuz{f})$ does not have the finite shadowing property for the systems $(X,f)$ considered in Examples~\ref{Exa_1:discrete}~and~\ref{Exa_2:connected}, but also for the expansive map $f(x)=2x$ defined in $X=\RR$ with the usual $|\cdot|$-metric: the maps of these systems have dense range but they are not topologically mixing. However, the systems $(\Fc_{0}(X),\fuz{f})$, $(\Fc_{S}(X),\fuz{f})$ and $(\Fc_{E}(X),\fuz{f})$ will have the shadowing property whenever $(X,f)$ is ``contractive enough'' (see Theorem~\ref{The:contractions} below). We start by recalling why contractions have the shadowing property:

\begin{proposition}\label{Pro:contractions}
	If a dynamical system $(X,f)$, or any of its extensions $(\Kc(X),\com{f})$ or $(\Fc_{\infty}(X),\fuz{f})$, is contractive, then $(X,f)$, $(\Kc(X),\com{f})$ and $(\Fc_{\infty}(X),\fuz{f})$ have the shadowing property.
\end{proposition}
\begin{proof}
	It is folklore that if $(X,f)$ is a contractive dynamical system for $\lambda \in [0,1[$, then $(X,f)$ has the shadowing property by taking $\delta_{\eps} := (1-\lambda)\eps$ for each $\eps>0$. For the sake of completeness we will include its brief proof. Actually, given any $d$-$\delta_{\eps}$-pseudo-trajectory $(x_j)_{j=0}^{\infty}$ one can consider the point $x:=x_0$, so that we have $d(f^0(x),x_0)=d(x_0,x_0)=0=(1-\lambda^0)\eps$. Arguing inductively, if we now assume that $d(f^j(x),x_j)<(1-\lambda^j)\eps$ for some non-negative integer $j\geq 0$, then we have that
	\begin{align*}
		d(f^{j+1}(x),x_{j+1}) &\leq d(f^{j+1}(x),f(x_j)) + d(f(x_j),x_{j+1}) \leq \lambda d(f^j(x),x_j) + \delta_{\eps} \\[5pt]
		&\leq \lambda (1-\lambda^j)\eps + (1-\lambda)\eps = (1-\lambda^{j+1})\eps < \eps.
	\end{align*}
	The whole result follows from the fact that $(X,f)$ is contractive if and only if so is $(\Kc(X),\com{f})$, but also if and only if so is $(\Fc_{\infty}(X),\fuz{f})$, as proved in \cite[Section~III.7]{Barnsley1988_book_fractals} and \cite[Proposition~5.7]{JardonSanSan2020_FSS_some}.
\end{proof}

Since the system $(X,f)$ considered in Example~\ref{Exa_2:connected} is contractive, it is clear that Proposition~\ref{Pro:contractions} cannot be obtained for $(\Fc_{0}(X),\fuz{f})$, $(\Fc_{S}(X),\fuz{f})$ and $(\Fc_{E}(X),\fuz{f})$ in general. However, if we add a particular boundedness assumption we reach our desired result:

\begin{theorem}\label{The:contractions}
	Let $f:X\longrightarrow X$ be a continuous map acting on a metric space $(X,d)$. If $(X,f)$ is a contractive system and the set $f^k(X)$ is bounded in $(X,d)$ for some $k \in \NN$, then the extended dynamical systems $(\Fc_{0}(X),\fuz{f})$, $(\Fc_{S}(X),\fuz{f})$ and $(\Fc_{E}(X),\fuz{f})$ have the shadowing property. In particular, the previous conclusion holds whenever $(X,f)$ is a contractive system and $(X,d)$ is bounded itself.
\end{theorem}
\begin{proof}
	Assume first that $(X,d)$ is a complete space. Thus, the Banach fixed-point theorem tells us that there exists a unique point, namely $x^* \in X$, fulfilling that $f(x^*)=x^*$. Let now $\eps>0$ be arbitrary but fixed. Since the system $(X,f)$ is contractive and $f(x^*)=x^*$, we can easily find two small enough positive values $\eps_1,\eps_2>0$ fulfilling that $0<\eps_2\leq \eps_1<\tfrac{\eps}{4}$ but also that
	\begin{equation}\label{eq:eps_1_2}
		f(\Bc_d(x^*,\eps_1+\eps_2)) \subset \Bc_d(x^*,\eps_1).% ... pick $\eps_1<\tfrac{\eps}{4}$ and then choose $\eps_2\leq\eps_1$ such that $\lambda(\eps_1+\eps_2)<\eps_1$ ...
	\end{equation}
	Using again that $(X,f)$ is contractive, that $f(x^*)=x^*$, and this time that $f^k(X)$ is bounded for some $k \in \NN$, i.e.\ that $\diam_d(f^k(X))<\infty$, we can easily find some positive integer $N \in \NN$ fulfilling that
	\begin{equation}\label{eq:f^j(Y).subset.eps_1}
		f^j(X) \subset \Bc_d(x^*,\eps_1) \quad \text{ for all } j \geq N.% ... note that $d(x^*,f^j(x)) = d(f^{(j-k)+k}(x^*),f^{(j-k)+k}(x)) \leq \lambda^{j-k} \cdot d(f^k(x^*),f^k(x)) \leq \lambda^{j-k} \cdot \diam_d(f^k(X)) < \eps_1$ for all $x \in X$ and all $j \geq N \geq k$ whenver $N \in \NN$ is big enough ... 
	\end{equation}
	Let $\rho \in \{ d_{0} , d_{S} , d_{E} \}$, fix $\delta_{\eps} := \tfrac{\eps_2}{N}$, and consider any $\rho$-$\delta_{\eps}$-pseudo-trajectory $(u^j)_{j=0}^{\infty}$ for $\fuz{f}$ in $\Fc(X)$. We will check that $(u^j)_{j=0}^{\infty}$ is $\rho$-$\eps$-shadowed by the $\fuz{f}$-orbit of $u := u^0$. We start by claiming that
	\begin{equation}\label{eq:claim.1}
		\rho(\fuz{f}^j(u),u^j) < j\delta_{\eps} \leq \eps_2 < \eps \quad \text{ for each } 0\leq j\leq N.
	\end{equation}
	Actually, we have that $\rho(\fuz{f}^{0}(u),u^0) = \rho(u^0,u^0) = 0$. Arguing inductively, if we now assume that the inequality $\rho(\fuz{f}^j(u),u^j)<j\delta_{\eps}$ holds for some $0\leq j< N$, by statement (a) of Lemma~\ref{Lem:almost.contraexpansive} we have that
	\begin{align*}
		\rho(\fuz{f}^{j+1}(u),u^{j+1}) &\leq \rho(\fuz{f}^{j+1}(u),\fuz{f}(u^j)) + \rho(\fuz{f}(u^j),u^{j+1}) \leq \rho(\fuz{f}^j(u),u^j) + \rho(\fuz{f}(u^j),u^{j+1}) \\[5pt]
		&< j\delta_{\eps} + \delta_{\eps} = (j+1)\delta_{\eps} \leq \eps_2 \leq \eps_1 < \tfrac{\eps}{4} < \eps.
	\end{align*}
	With \eqref{eq:claim.1} being proved, we now divide the proof in two cases to check that
	\begin{equation}\label{eq:claim.2}
		\rho(\fuz{f}^j(u),u^j)<\eps \quad \text{ for all } j > N.		
	\end{equation}
	\begin{enumerate}[--]
		\item \textbf{Case 1}: \textit{We have that $\rho \in \{ d_{0} , d_{S} \}$}. In this case we start by claiming that
		\begin{equation}\label{eq:u^j_0.subset.eps_1_2}
			u^j_0 \subset \Bc_d(x^*,\eps_1+\eps_2) \quad \text{ for all } j \geq N.
		\end{equation}
		Actually, for $j=N$ and since $[\fuz{f}^N(u)]_0 = f^N(u_0) \subset \Bc_d(x^*,\eps_1)$ by \eqref{eq:f^j(Y).subset.eps_1}, using \eqref{eq:claim.1} together with statement~(e) of Proposition~\ref{Pro:fuzzy.metrics} we have that
		\[
		d_H([\fuz{f}^N(u)]_0,u^N_0) \leq \rho(\fuz{f}^N(u),u^N) < N\delta_{\eps} = \eps_2,
		\]
		so that $u^N_0 \subset [\fuz{f}^N(u)]_0 + \eps_2 \subset \Bc_d(x^*,\eps_1) + \eps_2 \subset \Bc_d(x^*,\eps_1+\eps_2)$. Arguing inductively, if we now assume that $u^j_0 \subset \Bc_d(x^*,\eps_1+\eps_2)$ for some $j \geq N$, then by \eqref{eq:eps_1_2} we have that
		\[
		[\fuz{f}(u^j)]_0 = \com{f}(u^j_0) = f(u^j_0) \subset f(\Bc_d(x^*,\eps_1+\eps_2)) \subset \Bc_d(x^*,\eps_1).
		\]
		Using again Proposition~\ref{Pro:fuzzy.metrics} we obtain that $d_H([\fuz{f}(u^j)]_0,u^{j+1}_0) \leq \rho(\fuz{f}(u^j),u^{j+1}) < \delta_{\eps} \leq \eps_2$ and hence
		\[
		u^{j+1}_0 \subset [\fuz{f}(u^j)]_0 + \eps_2 \subset \Bc_d(x^*,\eps_1) + \eps_2 \subset \Bc_d(x^*,\eps_1+\eps_2).
		\]
		With \eqref{eq:u^j_0.subset.eps_1_2} being proved we can complete \textbf{Case 1} as follows. For each $j > N$ pick any pair of points $x_j \in [\fuz{f}^j(u)]_1$ and $y_j \in u^j_1$. It follows that $x_j \in [\fuz{f}^j(u)]_{\alpha} \subset [\fuz{f}^j(u)]_0 = f^j(u_0)$ and that $y_j \in u^j_\alpha \subset u^j_0$ for all $\alpha \in \II$, which implies, in particular, that $x_j,y_j \in \Bc_d(x^*,\eps_1+\eps_2)$ by \eqref{eq:f^j(Y).subset.eps_1} and \eqref{eq:u^j_0.subset.eps_1_2} respectively. In addition, note that given any level $\alpha \in \II$, then for every point $x \in [\fuz{f}^j(u)]_{\alpha} \subset f^j(u_0)$ we have that $x \in \Bc_d(x^*,\eps_1+\eps_2)$ by \eqref{eq:f^j(Y).subset.eps_1} and hence that
		\[
		d(x,u^j_{\alpha}) = \inf_{z \in u^j_{\alpha}} d(x,z) \leq d(x,y_j) \leq d(x,x^*) + d(x^*,y_j) < 2(\eps_1+\eps_2),
		\]		
		but also for every point $y \in u^j_{\alpha} \subset u^j_0$ we have that $y \in \Bc_d(x^*,\eps_1+\eps_2)$ by \eqref{eq:u^j_0.subset.eps_1_2} and hence that
		\[
		d(y,[\fuz{f}^j(u)]_{\alpha}) = \inf_{z \in [\fuz{f}^j(u)]_{\alpha}} d(y,z) \leq d(y,x_j) \leq d(y,x^*) + d(x^*,x_j) < 2(\eps_1+\eps_2).
		\]
		Since we had the inequalities $0<\eps_2\leq\eps_1<\tfrac{\eps}{4}$, we deduce that
		\begin{align*}
			\rho(\fuz{f}^j(u),u^j) &\leq d_{\infty}(\fuz{f}^j(u),u^j) = \sup_{\alpha\in\II} d_H([\fuz{f}^j(u)]_{\alpha},u^j_{\alpha}) \\[5pt]
			&\leq \max\left\{ \sup_{x \in [\fuz{f}^j(u)]_{\alpha}} d(x,u^j_{\alpha}) , \sup_{y \in u^j_{\alpha}} d(y,[\fuz{f}^j(u)]_{\alpha}) \right\} \leq 2(\eps_1+\eps_2) < \eps,
		\end{align*}
		which implies that \eqref{eq:claim.2} holds when $\rho \in \{ d_{0} , d_{S} \}$ and we have finished the proof for \textbf{Case 1}.
		
		\item \textbf{Case 2}: \textit{We have that $\rho=d_{E}$}. For this case, and without loss of generality, we will assume that we have chosen $\eps_2<1$. In addition, in the sequel we will repeatedly use the fact that
		\begin{equation}\label{eq:f(Y+delta).subset.f(Y)+delta}
			f(Y+\delta) \subset f(Y)+\delta \quad \text{ for every subset } Y \subset X \text{ and every positive value } \delta>0,
		\end{equation}
		which is a direct consequence of the fact that $(X,f)$ is contractive. We start by claiming that
		\begin{equation}\label{eq:end.subset.1}
			\eend(u^j) \subset \left(f^j(u_0)+j\delta_{\eps}\right)\times\II \ \cup \bigcup_{0\leq l\leq j-1} \left(f^l(X)+l\delta_{\eps}\right)\times[0,(l+1)\delta_{\eps}] \quad \text{ for all } 1 \leq j \leq N.
		\end{equation}
		Actually, since $\eend(u^0) = \eend(u) = \send(u) \ \cup \ X\times\{0\} \subset u_0\times\II \ \cup \ X\times\{0\}$, applying the map~$\fuz{f}$ it is not hard to check the inclusion $\eend(\fuz{f}(u^0)) \subset f(u_0)\times\II \ \cup \ X\times\{0\}$. Thus, using that $\com{d}_H(\eend(\fuz{f}(u^0)),\eend(u^1)) = d_{E}(\fuz{f}(u^0),u^1) < \delta_{\eps}$ by assumption, it follows that
		\[
		\eend(u^1) \subset \eend(\fuz{f}(u^0))+\delta_{\eps} \subset \left(f(u_0)+\delta_{\eps}\right)\times\II \ \cup \ X\times[0,\delta_{\eps}].
		\]
		This proves \eqref{eq:end.subset.1} for $j=1$. Arguing inductively, if we now assume that \eqref{eq:end.subset.1} holds for some positive integer $1\leq j< N$, applying the map $\fuz{f}$ to such formula it is not hard to check that
		\[
		\eend(\fuz{f}(u^j)) \subset f\left(f^j(u_0)+j\delta_{\eps}\right)\times\II \ \cup \bigcup_{0\leq l\leq j-1} f\left(f^l(X)+l\delta_{\eps}\right)\times[0,(l+1)\delta_{\eps}] \ \cup \ X\times\{0\}.
		\]
		Applying \eqref{eq:f(Y+delta).subset.f(Y)+delta} for the corresponding sets and positive values on the previous formula we get that
		\[
		\eend(\fuz{f}(u^j)) \subset \left(f^{j+1}(u_0)+j\delta_{\eps}\right)\times\II \ \cup \bigcup_{0\leq l\leq j-1} \left(f^{l+1}(X)+l\delta_{\eps}\right)\times[0,(l+1)\delta_{\eps}] \ \cup \ X\times\{0\}.
		\]
		Finally, since $\com{d}_H(\eend(\fuz{f}(u^j)),\eend(u^{j+1})) = d_{E}(\fuz{f}(u^j),u^{j+1}) < \delta_{\eps}$ by assumption, we have that
		\[
		\eend(u^{j+1}) \subset \eend(\fuz{f}(u^j))+\delta_{\eps} \subset \left(f^{j+1}(u_0)+(j+1)\delta_{\eps}\right)\times\II \ \cup \bigcup_{0\leq l\leq j} \left(f^l(X)+l\delta_{\eps}\right)\times[0,(l+1)\delta_{\eps}].
		\]
		With \eqref{eq:end.subset.1} being proved, we now claim that
		\begin{equation}\label{eq:end.subset.2}
			\eend(u^j) \subset \Bc_d(x^*,\eps_1+\eps_2)\times\II \ \cup \bigcup_{0 \leq l \leq N-1} \left(f^l(X)+l\delta_{\eps}\right)\times[0,(l+1)\delta_{\eps}] \quad \text{ for all } j \geq N.
		\end{equation}
		Actually, using that $f^N(u_0) \subset \Bc_d(x^*,\eps_1)$ by \eqref{eq:f^j(Y).subset.eps_1} and that $N\delta_{\eps}=\eps_2$, for $j=N$ we have that
		\begin{align*}
			\eend(u^N) &\overset{\eqref{eq:end.subset.1}}{\subset} \left(f^N(u_0)+N\delta_{\eps}\right)\times\II \ \cup \bigcup_{0\leq l\leq N-1} \left(f^l(X)+l\delta_{\eps}\right)\times[0,(l+1)\delta_{\eps}] \\
			& \ \ \subset \ \ \Bc_d(x^*,\eps_1+\eps_2)\times\II \ \cup \bigcup_{0 \leq l \leq N-1} \left(f^l(X)+l\delta_{\eps}\right)\times[0,(l+1)\delta_{\eps}].
		\end{align*}
		Arguing inductively, if we now assume that \eqref{eq:end.subset.2} holds for some $j\geq N$, applying the map $\fuz{f}$ to such formula it is not hard to check that
		\[
		\eend(\fuz{f}(u^j)) \subset f\left(\Bc_d(x^*,\eps_1+\eps_2)\right)\times\II \ \cup \bigcup_{0\leq l\leq N-1} f\left(f^l(X)+l\delta_{\eps}\right)\times[0,(l+1)\delta_{\eps}] \ \cup \ X\times\{0\}.
		\]
		Using \eqref{eq:eps_1_2} and applying $N$ times \eqref{eq:f(Y+delta).subset.f(Y)+delta} we thus have that
		\[
		\eend(\fuz{f}(u^j)) \subset \Bc_d(x^*,\eps_1)\times\II \ \cup \bigcup_{0\leq l\leq N-1} \left(f^{l+1}(X)+l\delta_{\eps}\right)\times[0,(l+1)\delta_{\eps}] \ \cup \ X\times\{0\},
		\]
		where the term $f^N(X)+(N-1)\delta_{\eps}$, corresponding to the element $l=N-1$ in the previous big union, fulfills that $f^N(X) \subset \Bc_d(x^*,\eps_1)$ by \eqref{eq:f^j(Y).subset.eps_1}. Hence, $f^N(X)+(N-1)\delta_{\eps} \subset \Bc_d(x^*,\eps_1+(N-1)\delta_{\eps})$, and the previous inclusion can be rewritten as follows
		\[
		\eend(\fuz{f}(u^j)) \subset \Bc_d(x^*,\eps_1+(N-1)\delta_{\eps})\times\II \ \cup \bigcup_{0\leq l\leq N-1} \left(f^l(X)+(l-1)\delta_{\eps}\right)\times[0,l\delta_{\eps}].
		\]		
		Finally, since $\com{d}_H(\eend(\fuz{f}(u^j)),\eend(u^{j+1})) = d_{E}(\fuz{f}(u^j),u^{j+1}) < \delta_{\eps}$ and $N\delta_{\eps}=\eps_2$, we have that
		\begin{align*}
			\eend(u^{j+1}) &\subset \eend(\fuz{f}(u^j))+\delta_{\eps} \subset \Bc_d(x^*,\eps_1+N\delta_{\eps})\times\II \ \cup \bigcup_{0\leq l\leq N-1} \left(f^l(X)+l\delta_{\eps}\right)\times[0,(l+1)\delta_{\eps}] \\
			&\subset \Bc_d(x^*,\eps_1+\eps_2)\times\II \ \cup \bigcup_{0\leq l\leq N-1} \left(f^l(X)+l\delta_{\eps}\right)\times[0,(l+1)\delta_{\eps}],
		\end{align*}
		With \eqref{eq:end.subset.2} being proved we can complete \textbf{Case 2} as follows. For each $j > N$ pick any pair of points $x_j \in [\fuz{f}^j(u)]_1$ and $y_j \in u^j_1$. Since $x_j \in [\fuz{f}^j(u)]_1 \subset [\fuz{f}^j(u)]_0 = f^j(u_0)$, using fact \eqref{eq:f^j(Y).subset.eps_1} we have that $x_j \in \Bc_d(x^*,\eps_1+\eps_2)$, that
		\[
		\{x_j\}\times\II \ \cup \ X\times\{0\} \subset \eend(\fuz{f}^j(u)) \subset \Bc_d(x^*,\eps_1)\times\II \ \cup \ X\times\{0\},
		\]
		and that $\Bc_d(x^*,\eps_1+\eps_2) \subset \{x_j\}+2(\eps_1+\eps_2)$ because $d(z,x_j) \leq d(z,x^*) + d(x^*,x_j) < 2(\eps_1+\eps_2)$ for each $z \in \Bc_d(x^*,\eps_1+\eps_2)$. In its turn, since $(l+1)N\delta_{\eps} \leq N\delta_{\eps} = \eps_2$ for all $0 \leq l \leq N-1$ and $\eps_2 < 1$ by assumption in \textbf{Case 2}, using \eqref{eq:end.subset.2} we necessarily have that $y_j \in \Bc_d(x^*,\eps_1+\eps_2)$, that
		\[
		\{y_j\}\times\II \ \cup \ X\times\{0\} \subset \eend(u^j) \subset \Bc_d(x^*,\eps_1+\eps_2)\times\II \ \cup \ X\times[0,\eps_2],
		\]
		and that $\Bc_d(x^*,\eps_1+\eps_2) \subset \{y_j\}+2(\eps_1+\eps_2)$ because $d(z,y_j) \leq d(z,x^*) + d(x^*,y_j) < 2(\eps_1+\eps_2)$ for all $z \in \Bc_d(x^*,\eps_1+\eps_2)$. From the previous inclusions, it follows that
		\begin{align*}
			\eend(\fuz{f}^j(u)) &\subset \Bc_d(x^*,\eps_1+\eps_2)\times\II \ \cup \ X\times\{0\} \subset \left(\{y_j\}+2(\eps_1+\eps_2)\right)\times\II \ \cup \ X\times\{0\} \\[5pt]
			&\subset \left(\{y_j\}\times\II \ \cup \ X\times\{0\}\right)+2(\eps_1+\eps_2) \subset \eend(u^j)+2(\eps_1+\eps_2),
		\end{align*}
		and that
		\begin{align*}
			\eend(u^j) &\subset \Bc_d(x^*,\eps_1+\eps_2)\times\II \ \cup \ X\times[0,\eps_2] \subset \left(\{x_j\}+2(\eps_1+\eps_2)\right)\times\II \ \cup X\times[0,2(\eps_1+\eps_2)] \\[5pt]
			&\subset \left(\{x_j\}\times\II \ \cup \ X\times\{0\}\right)+2(\eps_1+\eps_2) \subset \eend(\fuz{f}^j(u))+2(\eps_1+\eps_2).
		\end{align*}
		By statement (a) of Proposition~\ref{Pro:Hausdorff} we finally deduce that
		\[
		d_{E}(f^j(u),u^j) = \com{d}_H(\eend(f^j(u)),\eend(u^j)) \leq 2(\eps_1+\eps_2) < \eps,
		\]
		which implies that \eqref{eq:claim.2} holds for $\rho=d_{E}$ and we have finished the proof for \textbf{Case 2}.
	\end{enumerate}
	Assume now that $(X,d)$ is not a complete space. Consider its completion $(X^*,d^*)$ and identify $(X,d)$ as a proper dense subspace of $(X^*,d^*)$. Since $(X,d)$ is isometrically and densely embedded in $(X^*,d^*)$, if $f^*:(X^*,d^*)\longrightarrow(X^*,d^*)$ is the unique continuous extension of $f:(X,d)\longrightarrow(X,d)$ into $(X^*,d^*)$, then ${f^*}^k(X^*)$ is bounded in $(X^*,d^*)$. Moreover, the system $(X,f)$ is a subsystem of $(X^*,f^*)$, which can be checked to be contractive (recall that the extension of a Lipschitz map from a dense subspace is again Lipschitz with the same constant). As for $(X,d)$ and $(X^*,d^*)$, we can consider $\Fc(X)$ as a subset of $\Fc(X^*)$, and $(\Fc(X),\fuz{f})$ as a subsystem of $(\Fc(X^*),\fuz{f^*})$. By the first part of the proof we know that, given any metric $\rho^* \in \{ d_{0}^* , d_{S}^* , d_{E}^* \}$ and any value $\eps>0$ we can find $\delta_{\eps}>0$ such that every $\rho^*$-$\delta_{\eps}$-pseudo-trajectory $(u^j)_{j=0}^{\infty}$ for $\fuz{f^*}$ is $\rho^*$-$\eps$-shadowed by the $\fuz{f^*}$-orbit of $u^0$. Noticing that the metric space $(\Fc(X),\rho)$ is isometrically embedded in $(\Fc(X^*),\rho^*)$, where we are letting $\rho=\rho^*\res_{\Fc(X)}$, if we now choose a $\rho$-$\delta_{\eps}$-pseudo-trajectory $(u^j)_{j=0}^{\infty}$ in $\Fc(X)$ we have that
	\[
	\rho^*(\fuz{f^*}(u^j),u^{j+1}) = \rho^*(\fuz{f}(u^j),u^{j+1}) = \rho(\fuz{f}(u^j),u^{j+1}) < \delta_{\eps} \quad \text{ for all } 0\leq j< \infty,
	\]
	and we achieve that
	\[
	\rho(\fuz{f}^j(u^0),u^j) = \rho^*(\fuz{f}^j(u^0),u^j) = \rho^*(\fuz{f^*}^j(u^0),u^j) < \eps \quad \text{ for all } 0\leq j< \infty.
	\]
	Noticing that $d_{0}^*\res_{\Fc(X)}=d_{0}$, that $d_{S}^*\res_{\Fc(X)}=d_{S}$ and that $d_{E}^*\res_{\Fc(X)}=d_{E}$, the proof is finished.
\end{proof}

\begin{remark}
	Let us include two applications of Theorem~\ref{The:contractions}:
	\begin{enumerate}[(1)]
		\item Let $(\RR,f)$ be the system considered in Example~\ref{Exa_2:connected}, where $\RR$ is endowed with the usual $|\cdot|$-metric and $f(x)=\tfrac{1}{2}x$ for each $x \in \RR$. We proved there that $(\Fc_{0}(\RR),\fuz{f})$, $(\Fc_{S}(\RR),\fuz{f})$ and $(\Fc_{E}(\RR),\fuz{f})$ do not have the finite shadowing property. However, let $X \subset \RR$ be any $|\cdot|$-bounded set fulfilling that $f(X) \subset X$. It follows that the dynamical systems $(\Fc_{0}(X),\fuz{f})$, $(\Fc_{S}(X),\fuz{f})$ and $(\Fc_{E}(X),\fuz{f})$ have the shadowing property by Theorem~\ref{The:contractions}. In particular, given any positive value $r>0$ we could consider $X=\{0\}$, $X=]-r,0[$, $X=[-r,0[$, $X=]0,r[$, $X=]0,r]$, $X=\QQ \ \cap \ [-r,r]\setminus\{0\}$, or we could combine them by considering unions and intersections of these sets.
		
		\item Let $\CC$ be the set of complex numbers and let $|\cdot|$ denote the usual modulus there. Fix any $k \in \NN$ and, denoting by $i$ the imaginary constant and by $\exp(z)$ the exponential function, consider the subset $X=\{0\} \cup \{ r\exp(\tfrac{j\pi}{k}i) \ ; \ r>0 \text{ and } 1 \leq j \leq k \}$. In addition, define $d:X\times X\longrightarrow[0,\infty[$ as $d(z_1,z_2)=|z_1-z_2|$ when $\arg(z_1)=\arg(z_2)$ and $d(z_1,z_2)=|z_1|+|z_2|$ when $\arg(z_1)\neq\arg(z_2)$ for every pair $z_1,z_2 \in X$, where $\arg(0)=0$. Consider the map $f:X\longrightarrow X$ as $f(0)=0$, $f(r\exp(\tfrac{j\pi}{k}))=\tfrac{r}{2}\exp(\tfrac{(j+1)\pi}{k}i)$ for all $r>0$ and $1\leq j\leq k-1$, with $f(-r)=-\tfrac{r}{2}$ when $r \in \ ]0,1]$ and $f(-r)=-(1-\tfrac{3}{2^{l+2}})-\tfrac{r}{2^{2l+2}}$ when $r \in \ ]2^l,2^{l+1}]$ for some $l \in \NN_0$. Then, it is not hard to check that $(X,d)$ is an unbounded metric space, that $(X,f)$ is a contractive system, and that $f^j(X)$ is an unbounded set for $1 \leq j \leq k-1$ while $f^k(X)$ is a bounded set in $(X,d)$. By Theorem~\ref{The:contractions}, the corresponding systems $(\Fc_{0}(X),\fuz{f})$, $(\Fc_{S}(X),\fuz{f})$ and $(\Fc_{E}(X),\fuz{f})$ have the shadowing property.
	\end{enumerate}
\end{remark}

In view of Theorems~\ref{The:contraexpansive},~\ref{The:shadowing.E}~and~\ref{The:contractions}, we end the section posing a natural question:

\begin{problem}\label{Problem.1}
	Characterize the dynamical systems $(X,f)$ for which any of the extended dynamical systems $(\Fc_{0}(X),\fuz{f})$, $(\Fc_{S}(X),\fuz{f})$ or $(\Fc_{E}(X),\fuz{f})$ have the (resp.\ finite) shadowing property.
\end{problem}

\section{Conclusions}\label{Sec_6:conclusions}

In this paper, we have examined the interaction between several dynamical properties of a discrete system $(X,f)$ and its extension $(\Fc(X),\fuz{f})$ on the space of normal fuzzy sets $\Fc(X)$ and with respect to the supremum, Skorokhod, sendograph and endograph metrics. In particular, we have dealt with the notions of {\em contractive}, {\em expansive}, {\em expanding} and {\em positively expansive} systems, as well as with the notions of {\em chain recurrence}, {\em chain transitivity}, {\em chain mixing} and the {\em shadowing property}. Throughout this paper we have solved open problems posed in \cite{JardonSan2021_FSS_expansive,JardonSanSan2020_FSS_some}, provided counterexamples to \cite[Theorem~5]{BartollMaPeRo2022_AXI_orbit}, and established new results concerning the previous chain-type notions and the shadowing property.

We recall that this paper is a continuation of \cite{Lopez2026_IJFS_topological-I}, where, within the same fuzzy framework, we examined in depth several dynamical properties related to {\em transitivity}, {\em point-transitivity}, {\em recurrence}, and the notions of {\em Devaney chaos} and the {\em specification property}. Since our original motivation was to address all the dynamical properties considered in \cite{AlvarezLoPe2025_FSS_recurrence,BartollMaPeRo2022_AXI_orbit,JardonSan2021_FSS_expansive,JardonSan2021_IJFS_sensitivity,JardonSanSan2020_FSS_some,JardonSanSan2020_MAT_transitivity,MartinezPeRo2021_MAT_chaos} for the system $(\Fc(X),\fuz{f})$ with respect to the endograph metric~$d_{E}$, it only remains to investigate the notions of {\em Li-Yorke} and {\em distributional chaos} considered in \cite{MartinezPeRo2021_MAT_chaos}, and the notions of {\em sensitivity} studied in \cite{JardonSan2021_IJFS_sensitivity}. We have recently addressed these dynamical properties in \cite{AlvarezLo2026_IS_Li-Yorke}, in collaboration with other authors.

We would like to stress that our study focuses on the space $\Fc(X)$ of normal fuzzy sets, rather than on other commonly used fuzzy frameworks, since this is precisely the setting adopted in our main references \cite{AlvarezLoPe2025_FSS_recurrence,BartollMaPeRo2022_AXI_orbit,JardonSan2021_FSS_expansive,JardonSan2021_IJFS_sensitivity,JardonSanSan2020_FSS_some,JardonSanSan2020_MAT_transitivity,MartinezPeRo2021_MAT_chaos} mentioned above. For further details on the rationale behind this choice, we refer the reader to \cite[Section~5]{Lopez2026_IJFS_topological-I} and to the comments preceding Theorem~\ref{The:chains}. Moreover, and as noted in \cite{Lopez2026_IJFS_topological-I}, possible directions for future research include studying the dynamical properties of the system $(\Fc(X),\fuz{f})$ when $\Fc(X)$ is endowed with other metrics (see \cite{Huang2022_FSS_some}), as well as extending the analysis to the case where $X$ is a (not necessarily metrizable) uniform space (see \cite{JardonSanSan2023_IJFS_fuzzy,JardonSanSan2026_TA_transitivity}).

\section*{Funding}

The development of this paper was partially supported by
\[
\text{MCIN/AEI/10.13039/501100011033/FEDER, UE, Project PID2022-139449NB-I00.}
\]

\section*{Acknowledgments}

The author wants to thank Salud Bartoll, Nilson Bernardes Jr., F\'elix Mart\'inez-Jim\'enez, Alfred Peris, and Francisco Rodenas, for their valuable advice and careful readings of the manuscript. The author is also grateful to \'Angel Calder\'on-Villalobos and Manuel Sanchis for insightful discussions on the topic.

\newpage

{\footnotesize% We may try "\small" with "\\[-5mm]" ...

}

{\footnotesize
$\ $\\

\textsc{Antoni L\'opez-Mart\'inez}: Universitat Polit\`ecnica de Val\`encia, Institut Universitari de Matem\`atica Pura i Aplicada, Edifici 8E, 4a planta, 46022 Val\`encia, Spain. e-mail: alopezmartinez@mat.upv.es
}

\end{document}